\documentclass[a4paper,leqno,10pt]{amsart}

\usepackage{epsf,graphicx,epsfig}
\usepackage{amscd}
\usepackage{amsmath,latexsym,amssymb,amsthm}
\usepackage{mathabx}
\usepackage[nospace,noadjust]{cite}
\usepackage{textcomp}
\usepackage{setspace,cite}
\usepackage{lscape,fancyhdr,fancybox}
\usepackage{stmaryrd}
\usepackage[all,cmtip]{xy}
\usepackage{tikz}
\usepackage{cancel}
\usetikzlibrary{shapes,arrows,decorations.markings}
\usepackage{graphicx}

\usepackage{color}
\usepackage{url}
\usepackage{enumerate}
\usepackage[mathscr]{euscript}

  \theoremstyle{plain}

\swapnumbers
    \newtheorem{thm}{Theorem}[section]
    \newtheorem{prop}[thm]{Proposition}
     
   \newtheorem{lemma}[thm]{Lemma}

    \newtheorem{subsec}[thm]{}
\theoremstyle{definition}
    \newtheorem{defn}[thm]{Definition}
        \newtheorem{remark}[thm]{Remark}
    \newtheorem{exam}[thm]{Example}

\theoremstyle{remark}

\title{}
\author{}
\date{}
\usepackage{amssymb}

\usepackage{hyperref}
\hypersetup{
	colorlinks,
	citecolor=blue,
	filecolor=black,
	linkcolor=blue,
	urlcolor=black
}

\begin{document}
\title{Hom-associative algebras up to homotopy}
\author{Apurba Das}
\email{apurbadas348@gmail.com}
\address{Stat-Math Unit,
Indian Statistical Institute, Kolkata 700108, West Bengal, India.}
\thispagestyle{empty}

\begin{abstract}
A hom-associative algebra is an algebra whose associativity is twisted by an algebra homomorphism. In this paper, we introduce a strongly homotopy version of hom-associative algebras ($HA_\infty$-algebras in short) on a graded vector space. We describe $2$-term $HA_\infty$-algebras in details. In particular, we study `skeletal' and `strict' $2$-term $HA_\infty$-algebras. We also introduce hom-associative $2$-algebras as categorification of hom-associative algebras. The category of $2$-term $HA_\infty$-algebras and the category of hom-associative $2$-algebras are shown to be equivalent. An appropriate skew-symmetrization of $HA_\infty$-algebras give rise to $HL_\infty$-algebras introduced by Sheng and Chen. Finally, we define a suitable Hochschild cohomology theory for $HA_\infty$-algebras which control the deformation of the structures.
\end{abstract}

\maketitle
\tableofcontents

\vspace{0.2cm}
\section{Introduction}
Homotopy algebras was first appeared in the work of Stasheff \cite{sta} in the recognition of loop spaces. More precisely, a connected topological space $X$ is a $1$-fold loop space if and only if it admits a structure of an $A_\infty$-space. The algebraic analouge of an $A_\infty$-space is called an $A_\infty$-algebra (strongly homotopy associative algebra). An $A_\infty$-algebra consists of a graded vector space $A$ together with a collection of multilinear maps $\mu_k : A^{\otimes k } \rightarrow A$ of degree $k-2$, for $k \geq 1$, satisfying certain conditions. A first example of an $A_\infty$-algebra is a differential graded associative algebra in which $\mu_k = 0$, for $k \geq 3$. 
An equivalent description of an $A_\infty$-algebra structure on $A$ is given by a square-zero coderivation on the tensor coalgebra of the suspension $sA$.
The Lie analouge of an $A_\infty$-algebra is 
called an $L_\infty$-algebra (strongly homotopy Lie algebra) which was first appeared in a supporting role in deformation theory \cite{lada-stasheff, lada-markl}. $L_\infty$-algebras also play a fundamental role in the Kontsevich's proof of the deformation quantization of arbitrary Poisson manifolds.

In this paper, we deal with certain type of algebras, called hom-type algebras. In these algebras, the identities defining the structures are twisted by homomorphisms. Recently, hom-type algebras have been studied by many authors. The notion of hom-Lie algebras was first introduced by Hartwig, Larsson and Silvestrov \cite{hls}. Hom-Lie algebras appeared in examples of q-deformations of the Witt and Virasoro algebras. Other type of algebras (e.g. associative, Leibniz, Poisson, Hopf,...) twisted by homomorphisms have also been studied. See \cite{makh-sil, makh-sil3} (and references there in) for more details. Our main objective of this paper is the notion of hom-associative algebra introduced by Makhlouf and Silvestrov \cite{makh-sil} and was further studied in \cite{makh-sil3, amm-ej-makh, makh-sil2, yau}. A hom-associative algebra is an algebra whose associativity is twisted by an algebra homomorphism. More precisely, a hom-associative algebra is an algebra $(A, \mu)$ and an algebra homomorphism $\alpha : A \rightarrow A$ satisfying
\begin{center}
$ \mu (   \alpha (a), \mu(b,c) ) = \mu( \mu(a,b), \alpha(c)  ),$
\end{center}
for all $a,b,c \in A$. When $\alpha = ~$identity, one recover the classical notion of associative algebras as a subclass. To extend the formal deformation theory from associative algebras to hom-associative algebras, the authors in \cite{amm-ej-makh, makh-sil2} introduce a Hochschild type cohomology theory (suitably twisted by $\alpha$) for hom-associative algebras. Recently, the present author shows that like the classical associative case \cite{gers0}, the Hochschild cohomology of a hom-associative algebra carries a Gerstenhaber algebra structure \cite{das}. This Gerstenhaber structure on cohomology is in fact induced from a homotopy $G$-algebra structure on the Hochschild cochain groups \cite{das2}.

The notion of strongly homotopy hom-Lie algebras or $HL_\infty$-algebras was introduced by Sheng and Chen as a hom-analouge of $L_\infty$-algebras \cite{sheng-chen}. The aim of the present paper is to introduce the hom-analouge of $A_\infty$-algebras or strongly homotopy hom-associative algebras ($HA_\infty$-algebras in short). More precisely, an $HA_\infty$-algebra structure on a graded vector space $A$ consists of
multilinear maps $\mu_k : A^{\otimes k} \rightarrow A$ of degree $k-2$, for $k \geq 1$, a linear map $\alpha : A \rightarrow A$ of degree $0$ which commute with all $\mu_k$'s and satisfying certain conditions (cf. Definition \ref{ha}). We define $HA_\infty[1]$-algebras as equivalent notion of $HA_\infty$-algebras by a degree shift (cf. Definition \ref{ha1}). We also give an equivalent description of $HA_\infty$-algebras in terms of suitable coderivations. Given a graded vector space $V$ and a linear map $\alpha : V \rightarrow V$ of degree $0$, we consider
\begin{center}
$\text{Coder}^p_\alpha (TV) = \{ D = \sum_{n \geq 1} \widetilde{\varrho_n} |~ \varrho_n : V^{\otimes n} \rightarrow V \text{ is a map of degree } p \text{ satisfying } \alpha \circ \varrho_n = \varrho_n \circ \alpha^{\otimes n} \},$
\end{center}
where each $\widetilde{\varrho_n} : TV \rightarrow TV$ is obtained from $\varrho_n$ by a suitable coderivation rule (cf. Lemma \ref{lifting-lemma}). The graded space $\text{Coder}^\bullet_\alpha (TV)$ carries a graded Lie algebra structure under the graded commutator of coderivations. We show that an $HA_\infty$-algebra structure on a graded vector space $A$ with respect to a degree $0$ linear map $\alpha$ is equivalent to an element $D \in \text{Coder}^{-1}_\alpha (TV)$ whose square is zero, where $V= s A$ (cf. Theorem \ref{coder-ha-inf}). Like classical cases of $A_\infty$-algebras, we also prove homotopy transfer theorems for $HA_\infty$-algebras (Theorems \ref{htt-1}, \ref{htt-2}). We remark that similar notions of strongly homotopy hom-associative algebras have been defined in \cite{alo-chat, yau3}. Although, our definition is much fruitful and perfectly fits with the cohomology theory of hom-associative algebras (see the contents of Section \ref{sec4}).

Note that, explicit description of $2$-term $L_\infty$-algebras have been studied in \cite{baez-crans} and they are related to categorification of Lie algebras. In this paper, we describe $2$-term $HA_\infty$-algebras in more details. More precisely, a $2$-term $HA_\infty$-algebra is given by a complex $A:= (A_1 \xrightarrow{d} A_0)$ together with maps $\mu_2 : A_i \otimes A_j \rightarrow A_{i+j}$ and $\mu_3 : A_0 \otimes A_0 \otimes A_0 \rightarrow A_1$, a chain map $\alpha : A \rightarrow A$ satisfying a set of axioms (cf. Definition \ref{defn-2ha-inf}). 
We denote the category of $2$-term $HA_\infty$-algebras by ${\bf 2HA_\infty}$. Particular cases are given by `skeletal' and `strict' $2$-term $HA_\infty$-algebras. Skeletal $2$-term $HA_\infty$-algebras are given by the condition $d=0$ and strict $2$-term $HA_\infty$-algebras are given by the condition $\mu_3 = 0$. We show that skeletal algebras are classified by third Hochschild cohomology of hom-associative algebras (cf. Theorem \ref{skeletal-2}). More generally, we prove that an $n$-term $HA_\infty$-algebras whose underlying graded vector space is concentrated in degrees $0, n-1$ and the zero differential are classified by $(n+1)$-th Hochschild cohomology of hom-associative algebras (cf. Theorem \ref{skeletal-n}). We introduce crossed module of hom-associative algebras and prove that strict algebras correspond to crossed modules of hom-associative algebras (cf. Theorem \ref{strict-crossed-mod}).

A $2$-vector space is a category $C$ with a vector space of objects $C_0$ and a vector space of morphisms $C_1$, such that all the structure maps are linear. It is known that the category of $2$-term complexes and the category of $2$-vector spaces are equivalent \cite{baez-crans}. 
We introduce hom-associative $2$-algebras as categorification of hom-associative algebras (cf. Definition \ref{hom-ass-2-alg-defn}). We denote the category of hom-associative $2$-algebras by ${\bf HAss2}$. Following the proof of \cite{baez-crans}, we prove that the categories ${\bf 2HA_\infty}$ and ${\bf HAss2}$ are equivalent (cf. Theorem \ref{cat-are-equiv}).

Next, we define left (right) modules and bimodules over $HA_\infty$-algebras. It turns out that any $HA_\infty$-algebra is a module over itself. We define a suitable Hochschild cohomology of an $HA_\infty$-algebra with coefficients in itself. Let $(A, \mu_k, \alpha)$ be an $HA_\infty$-algebra with the corresponding square-zero  coderivation $D \in \text{Coder}^{-1}_\alpha (TV)$, where $V= sA$. Then the Hochschild cochains are defined as
\begin{center}
$C^\bullet_\alpha (A, A) := \text{Coder}^{-(\bullet - 1)}_\alpha (TV).$
\end{center}
The coderivation $D$ induces a differential $\delta_\alpha (-) := [D, -]$ on $C^\bullet_\alpha (A, A)$. The cohomology groups are denoted by $H^\bullet_\alpha (A,A)$. 
It is known that the Hochschild cohomology of an $A_\infty$-algebra carries a Gerstenhaber structure \cite{getz-jon}. A similar method can be adapted to show that the Hochschild cohomology of an $HA_\infty$-algebra inherits a Gerstenhaber structure (cf. Theorem \ref{cohomology-g-alg}). Motivated from the formal deformation of hom-associative algebras \cite{amm-ej-makh, makh-sil2}, we study formal deformation of $HA_\infty$-algebras. We show that the deformation is controlled by the Hochschild cohomology of $HA_\infty$-algebras. Our main results about deformations are similar to classical cases. More precisely, the vanishing of the second Hochschild cohomology $H^2_\alpha (A,A)$ implies that the $HA_\infty$-algebra $A$ is rigid (cf. Theorem \ref{2-zero-rigid}) and the vanishing of the third cohomology $H^3_\alpha (A,A)$ allows one to extend a deformation of order $n$ to a deformation of order $n+1$ (cf. Theorem \ref{3-zero-extension}).

Finally, we deal with $HL_\infty$-algebras introduced by Sheng and Chen \cite{sheng-chen}. We define $HL_\infty [1]$-algebras as equivalent notion of $HL_\infty$-algebras by a degree shift. Like $HA_\infty$-algebras, we give an equivalent description of $HL_\infty$-algebras in terms of suitable coderivations (cf. Theorem \ref{hl-coder-thm}). Using this result, we show that a suitable skew-symmetrization of $HA_\infty$-algebras give rise to $HL_\infty$-algebras (cf. Theorem \ref{ha-comm-hl}). At the end, we define module over $HL_\infty$-algebras and Chevalley-Eilenberg cohomology. We also remark about the deformation of $HL_\infty$-algebras.

In the appendix, we introduce some generalizations of $HA_\infty$-algebras and $HL_\infty$-algebras. Namely,
we introduce $HA_\infty$-categories (categorical model of $HA_\infty$-algebras), $HA_\infty$-coalgebras (hom-coalgebras up to homotopy) and $HP_\infty$-algebras (hom-Poisson algebras up to homotopy).

Throughout the paper, we assume that the (graded) linear map $\alpha : A \rightarrow A$ is not identically zero. We also assume that the underlying field has characteristic zero.

\noindent {\bf Sign convention.} In order to deal with graded vector spaces, it is important to keep track of signs. The sign convention is that whenever two symbols of degrees $p$ and $q$, respectively, are interchanged, we multiply by a sign $(-1)^{pq}$. For a graded vector space $V = \oplus V_i$ and graded homogeneous indeterminates $a_1, \ldots, a_n \in V$, a permutation $\sigma \in S_n$, the Koszul sign $\epsilon (\sigma) = \epsilon (\sigma; a_1, \ldots, a_n)$ is given by
\begin{center}
$a_1 \wedge \cdots \wedge a_n = \epsilon (\sigma)~ a_{\sigma (1)} \wedge \cdots \wedge a_{\sigma (n)},$
\end{center}
where $\wedge$ is the product on the free graded commutative algebra generated by $\{a_1, \ldots, a_n\}$. Note that, $\epsilon (\sigma)$ does not involve the usual sign $(-1)^\sigma$ of the permutation. Denote $\chi (\sigma) = \chi (\sigma ; a_1, \ldots, a_n)$ by $\chi (\sigma) = (-1)^\sigma \epsilon (\sigma)$.

For any graded vector spaces $V$ and $W$, a graded linear map $f : V^{\otimes n} \rightarrow W$ (of degree $r$) is called symmetric if
\begin{center}
$ f (a_{\sigma(1)} \otimes \cdots \otimes a_{\sigma(n)}) = \epsilon (\sigma)~ f (a_1 \otimes \cdots \otimes a_n)$
\end{center}
and is called skew-symmetric if
\begin{center}
$ f (a_{\sigma(1)} \otimes \cdots \otimes a_{\sigma(n)}) = \chi (\sigma)~ f (a_1 \otimes \cdots \otimes a_n).$
\end{center}

\medskip

\noindent {\bf Organization.} In Section \ref{sec2}, we recall some basic definitions on hom-associative and hom-Lie algebras. on hom-associative algebra and hom-Lie algebra in the differential graded context. In Section \ref{sec3}, we introduce $HA_\infty$-algebras and other equivalent descriptions.
We also describe homotopy tranfer theorems for $HA_\infty$-algebras. In Section \ref{sec4}, we focus on $2$-term $HA_\infty$-algebras. In particular, we discuss about  skeletal $2$-term $HA_\infty$-algebras and strict $2$-term $HA_\infty$-algebras. Section \ref{sec5} concerns about hom-associative $2$-algebras. In Section \ref{sec6}, we discuss cohomology and deformations of $HA_\infty$-algebras.
Finally, in Section \ref{sec7}, we visit $HL_\infty$-algebras.

\section{Preliminaries}\label{sec2}
In this section, we recall hom-associative algebras and hom-Lie algebras in the differential graded perspectives. For more details on hom-algebras, see \cite{makh-sil, hls, makh-sil2, makh-sil3, amm-ej-makh, das, das2, yau}.

\subsection{DG hom-associative algebras}\label{subsec-two-one}

\begin{defn}
	A hom-associative algebra is a triple $(A, \mu, \alpha)$ consists of a vector space $A$ together with a bilinear map
	$\mu : A \otimes A \rightarrow A$ and a linear map $\alpha : A \rightarrow A$ satisfying $\alpha ( \mu( a, b)) = \mu (\alpha (a) , \alpha (b))$ and the following hom-associativity condition
	\begin{align}\label{hom-ass-cond}
	\mu ( \alpha (a) , \mu ( b , c) ) = \mu ( \mu (a , b) , \alpha (c)), ~ \text{ for all } a, b, c \in A.
	\end{align}
\end{defn}

In \cite{amm-ej-makh} the authors called such a hom-associative algebra `multiplicative'. By a hom-associative algebra, they mean a triple $(A, \mu, \alpha)$ of a vector space $A$, a bilinear map $\mu : A \otimes A \rightarrow A$ and a linear map $\alpha : A \rightarrow A$ satisfying condition (\ref{hom-ass-cond}). When $\alpha =$ identity, in any case, one gets the definition of a classical associative algebra.

Next, we recall the definition of Hochschild type cohomology for hom-associative algebras with coefficients in a hom-bimodule. Let $(A, \mu, \alpha)$ be a hom-associative algebra. A hom-bimodule over it consists of a vector space $M$ together with a linear map $\beta : M \rightarrow M$ and maps $\cdot : A \otimes M \rightarrow M$ and $\cdot : M \otimes A \rightarrow M$ satisfying $\beta (a \cdot m) = \alpha (a) \cdot \beta (m), ~~ \beta (m \cdot a) = \beta (m) \cdot \alpha (a)$ and the followings are hold
\begin{align*}
\alpha (a) \cdot (b \cdot m) &= \mu (a, b) \cdot \beta (m), \\
\beta (m) \cdot \mu (a, b) &= (m \cdot a) \cdot \alpha (b), \\
\alpha (a) \cdot (m \cdot b) &= (a \cdot m) \cdot \alpha (b),
\end{align*}
for all $a, b \in A$ and $m \in M$. The above three relations can be expressed in terms of the commutativity of the following diagrams
\[
\xymatrixrowsep{0.5in}
\xymatrixcolsep{0.3in}
\xymatrix{
A \otimes A \otimes M \ar[r]^{\mu \otimes \beta} \ar[d]_{\alpha \otimes \cdot} & A \otimes M \ar[d]^{\cdot} & M \otimes A \otimes A \ar[r]^{\beta \otimes \mu} \ar[d]_{\cdot \otimes \alpha} & M \otimes A \ar[d]^{\cdot} & A \otimes M \otimes A \ar[r]^{\alpha \otimes \cdot} \ar[d]_{\cdot \otimes \alpha} & A \otimes M \ar[d]^{\cdot} \\
A \otimes M \ar[r]_{\cdot} & M & M \otimes A \ar[r]_{\cdot} & M & M \otimes A \ar[r]_{\cdot} & M 
}
\]

It follows that $M= A$ is a hom-bimodule over $A$ with respect to the linear map $\alpha : A \rightarrow A$.
Given a hom-bimodule $(M, \beta, \cdot)$, one can define a Hochschild type cohomology of $(A, \mu, \alpha)$ with coefficients in the hom-bimodule. The cochain complex is given by
$\big( C^\bullet_{\alpha, \beta} (A,M), \delta_{\alpha, \beta} \big)$ where
\begin{center}
$C^n_{\alpha, \beta} (A, M) := \{   f : A^{\otimes n} \rightarrow M |~f (\alpha (a_1), \ldots, \alpha (a_n) ) = \beta (  f (a_1, \ldots, a_n)  ), \text{ for all } a_i \in A \}$
\end{center}
and 
\begin{align*} 
\delta_{\alpha, \beta} (f) (a_1, \ldots, a_{n+1}) =~&  \alpha^{n-1}(a_1) \cdot f(a_2, \ldots , a_{n+1})  \\
~& + \sum_{i=1}^{n} (-1)^i f \big( \alpha(a_1), \ldots, \alpha (a_{i-1}), \mu (a_i , a_{i+1}), \alpha (a_{i+2}), \ldots, \alpha (a_{n+1}) \big) \\
~& + (-1)^{n+1} f (a_1, \ldots, a_n) \cdot \alpha^{n-1} (a_{n+1}).
\end{align*} 
The cohomology of this complex is called the Hochschild cohomology of $A$ with coefficients in the hom-bimodule $(M, \beta, \cdot)$ and the cohomology groups are denoted by $H^\bullet_{\alpha, \beta} (A, M)$.
When $M= A$ and $\beta = \alpha$, we denote the Hochschild cochain complex by $\big( C^\bullet_\alpha (A,A), \delta_\alpha \big)$ and the cohomology is denoted by $H^\bullet_\alpha (A, A)$. This is called the Hochschild cohomology of $(A, \mu, \alpha)$. Like classical case, these cohomology theory control the deformation of hom-associative algebras \cite{amm-ej-makh}. Recently, the present author showed that the cohomology $H^\bullet_\alpha (A, A)$ inherits a Gerstenhaber algebra structure \cite{das, das2}.

A graded hom-associative algebra consists of a graded vector space $A = \oplus A_i$ together with a bilinear map $\mu : A \otimes A \rightarrow A$ of degree $0$ and a linear map $\alpha : A \rightarrow A$ of degree $0$ satisfying $\alpha (  \mu (a,b)) = \mu (   \alpha (a), \alpha (b))$ and the hom-associativity condition
\begin{center}
$\mu ( \alpha (a) , \mu ( b , c) ) = \mu ( \mu (a , b) , \alpha (c)), ~ \text{ for all } a, b, c \in A.$
\end{center}

A differential graded (DG) hom-associative algebra is a graded hom-associative algebra $(A = \oplus A_i, \mu, \alpha)$ together with a differential $d : A \rightarrow A$ of degree $-1$ satisfying $\alpha \circ d = d \circ \alpha$ and
\begin{center}
$d (   \mu (a,b)) = \mu (da, b) + (-1)^{|a|} \mu (a, db), ~\text{ for all } a, b \in A.$
\end{center}

\subsection{DG hom-Lie algebras}

\begin{defn}
	A hom-Lie algebra is a triple $(L, [-,-], \alpha)$ consists of a vector space $L$ together with a skew-symmetric bilinear map $[-,-] : L \otimes L \rightarrow L$ and a linear map $\alpha : L \rightarrow L$ satisfying $\alpha([a,b])= [\alpha(a), \alpha (b)]$ and the following hom-Jacobi identity
\begin{center}	
	$[[a,b], \alpha(c)] + [ [b,c], \alpha(a)] + [[c,a], \alpha(b)] = 0, \text{ for all } a,b,c \in L.$
	\end{center}
\end{defn}
When $\alpha = $identity, one gets the definition of a Lie algebra. A module over a hom-Lie algebra $(L, [-,-], \alpha)$ consists of a vector space $M$ together with a linear map $\beta : M \rightarrow M$ and a map $[-,-] : L \otimes M \rightarrow M$ satisfying $\beta [a,m] = [\alpha (a), \beta (m)]$ and
\begin{center}
$[[a,b], \beta(m)] = [\alpha (a), [b, m]] - [\alpha (b), [a,m]],~\text{ for all }a,b \in L,~ m \in M.$
\end{center}
It is clear that an hom-Lie algebra is a module over itself. One may also define a Chevalley-Eilenberg cohomology for hom-Lie algebras with coefficients in a module \cite{amm-ej-makh}.

A graded hom-Lie algebra consists of a graded vector space $L = \oplus L_i$ together with a graded skew-symmetric bilinear map $[-,-]: L \otimes L \rightarrow L$ of degree $0$ and a linear map $\alpha : L \rightarrow L$ of degree $0$ satisfying $\alpha ([a,b]) = [\alpha(a), \alpha (b)]$ and the following graded hom-Jacobi identity 
\begin{center}
$(-1)^{|a||c|} ~[[a,b], \alpha(c)] +  (-1)^{|b||a|}~  [ [b,c], \alpha(a)] +
(-1)^{|c||b|}~ [[c,a], \alpha(b)] = 0, \text{ for all } a,b,c \in L.$
\end{center}

A differential graded (DG) hom-Lie algebra is a graded hom-Lie algebra $(L = \oplus L_i, [-,-], \alpha)$ together with a differential $d : L \rightarrow L$ of degree $-1$ satisfying $\alpha \circ d = d \circ \alpha$ and
\begin{center}
$d [a,b] = [da, b] + (-1)^{|a|} [a, db], \text{ for all } a,b \in L.$
\end{center}

\section{Strongly homotopy hom-associative algebras}\label{sec3}

The aim of this section is to introduce hom-analouge of $A_\infty$-algebras or strongly homotopy hom-associative algebras ($HA_\infty$-algebras in short). We also prove homotopy tranfer theorems for $HA_\infty$-algebras.

\subsection{$HA_\infty$-algebras}
\begin{defn}\label{ha}
	An $HA_\infty$-algebra is a graded vector space $A = \oplus A_i$ together with
	\begin{itemize}
		\item[(i)] a collection $\{ \mu_k | ~ 1 \leq k < \infty \}$ of linear maps
		$\mu_k : A^{\otimes k} \rightarrow A$ with deg~$(\mu_k) = k-2,$
		\item[(ii)] a linear map $\alpha : A \rightarrow A$ of degree $0$ with
$$	\alpha \big(	\mu_k (a_1, \ldots, a_k)  \big) = \mu_k \big( \alpha (a_1), \ldots, \alpha (a_k) \big)$$
	\end{itemize}
such that for all $n \geq 1$,
\begin{align}\label{ha-eqn}
\sum_{i+j = n+1}^{} \sum_{\lambda =1}^{j} (-1)^{\lambda (i+1) + i (|a_1| + \cdots + |a_{\lambda -1 }|)} ~ \mu_{j} \big(  \alpha^{i-1}a_1, \ldots, \alpha^{i-1} a_{\lambda -1}, \mu_i ( a_{\lambda}, \ldots, a_{\lambda + i-1}),\\
 \alpha^{i-1} a_{\lambda + i}, \ldots, \alpha^{i-1} a_n   \big) = 0 \nonumber,
\end{align}
for $a_i \in A_{|a_i|}, ~ 1 \leq i \leq n$, or, equivalently,
\begin{align}\label{some-ha-identity}
\sum_{r + s + t = n, r , t \geq 0, s \geq 1}^{} (-1)^{rs + t}~ \mu_{r + 1 + t} ( (\alpha^{s-1})^{\otimes r} \otimes \mu_s \otimes (\alpha^{s-1})^{\otimes t}  ) = 0.
\end{align}
\end{defn}

An $HA_\infty$-algebra as above is denoted by $(A, \mu_k, \alpha).$ When $\alpha = \text{identity}$, one gets the definition of an $A_\infty$-algebra introduced by Stasheff \cite{sta}. When $A$ is a vector space considered as a graded vector space concentrated in degree $0$, we get hom-associative algebras \cite{makh-sil}.

\medskip

The above definition of an $HA_\infty$-algebra has the following consequences. For $n=1$, we get
$\mu_1^2 = 0,$
which means that the degree $-1$ map $\mu_1 : A \rightarrow A$ is a differential. Therefore, $(A, \mu_1)$ is a chain complex. For $n=2$, we get
\begin{center}
$ \mu_1 \big( \mu_2 (a, b) \big) = \mu_2 \big( \mu_1(a), b \big) + (-1)^{|a|}~ \mu_2 \big(a, \mu_1(b) \big), ~~~ \text{ for } a, b \in A.$
\end{center}
It says that the differential $\mu_1$ is a graded derivation for the product $\mu_2$.
For $n=3$, we have
\begin{align*}
&\mu_2 \big(  \alpha (a), \mu_2 (b, c) \big) - \mu_2 \big( \mu_2 (a, b), \alpha (c) \big)\\
& = - \bigg\{ \mu_1 \big(   \mu_3 (a, b, c) \big) + \mu_3 \big(  \mu_1 (a), b, c \big) + (-1)^{|a|} \mu_3 (a, \mu_1 (b), c) + (-1)^{|a| + |b|} \mu_3 (a, b, \mu_1 (c))  \bigg\}.
\end{align*}
This shows that the product $\mu_2$ does not satisfy (in general) the graded hom-associativity condition. However, it does satisfy up to a term involving $\mu_3$. Similarly, for higher $n$, we get higher coherence laws that $\mu_k$'s must satisfy. It is now easy to see that a graded hom-associative algebra is an $HA_\infty$-algebra with $\mu_k = 0$, for $k \neq 2$, and, 
a DG hom-associative algebra is an $HA_\infty$-algebra with $\mu_k = 0$, for $k \geq 3$.

It also follows from the above observation that the homology $H_* (A) = H_* (A, \mu_1)$ of a $HA_\infty$-algebra $(A, \mu_k, \alpha)$ carries a graded hom-associative algebra. 

\begin{exam}
Let $(A, \mu_k)$ be an $A_\infty$-algebra and $\alpha : A \rightarrow A$ be a degree $0$ map satisfying $\alpha \circ \mu_k = \mu_k \circ \alpha^{\otimes k}$, for all $k \geq 1$. Such a map $\alpha$ is called a strict morphism. If $\alpha: A \rightarrow A$ is a strict morphism, then $(A, \alpha^{k-1} \circ \mu_k, \alpha)$ is an $HA_\infty$-algebra.
\end{exam}

\begin{remark}\label{rem-ha-derivation}
Let $(A, \mu_k, \alpha)$ be an $HA_\infty$-algebra. Then $\mu_1 : A \rightarrow A$ is a differential, denoted by $d$. For $n \geq 2$, the derivation of $\mu_n : A^{\otimes n} \rightarrow A$ is given by
\begin{align*}
\partial (\mu_n) = [d, \mu_n] =~& d \circ \mu_n - (-1)^{n-2} \mu_n \circ d \\
=~& d \circ \mu_n - (-1)^{n-2} \mu_n \big(  (d, \text{id}, \ldots, \text{id}) + (\text{id}, d, \text{id}, \ldots, \text{id}) + \cdots + (\text{id}, \ldots, \text{id}, d)   \big).
\end{align*}
Then the relations (\ref{some-ha-identity}) for $n \geq 2$ becomes
\begin{align}\label{ha-derivation-defn}
\partial (\mu_n) =  \begin{cases}
~0, &  \mbox{   if } n = 2,\\
~- \sum_{r + s + t = n, r, t \geq 0, 1< s < n}^{} (-1)^{rs + t} ~ \mu_{r + 1 + t} ( (\alpha^{s-1})^{\otimes r} \otimes \mu_s \otimes (\alpha^{s-1})^{\otimes t} ), & \mbox{ if } n \geq 3.
\end{cases}
\end{align}
\end{remark}

\medskip

Motivated from the classical case, we also introduce an equivalent notion of $HA_\infty [1]$-algebra.
\begin{defn}\label{ha1}
An $HA_\infty [1]$-algebra structure on a graded vector space $V$ consists of degree $-1$ multilinear maps $\varrho_k : V^{\otimes k} \rightarrow V$, $k \geq 1$, and a linear map $\alpha : V \rightarrow V$ of degree $0$ such that
\begin{center}
$ \alpha \big( \varrho_k ( v_1, \ldots, v_k ) \big) = \varrho_k \big( \alpha(v_1), \ldots, \alpha(v_k) \big)$
\end{center}
and for all $n \geq 1$,
\begin{align}
\sum_{i+j = n+1}^{} \sum_{\lambda=1}^{j} (-1)^{|v_1| + \cdots + |v_{\lambda -1}|}~ \varrho_j \big( \alpha^{i-1 }v_1, \ldots, \alpha^{i-1}v_{\lambda -1}, \varrho_i (v_\lambda, \ldots, v_{\lambda + i-1}), \ldots, \alpha^{i-1} v_n  \big) = 0,
\end{align}
for $v_1, \ldots, v_n \in V$, or, equivalently,
\begin{align}
\sum_{r+s + t = n, r, t \geq 0, s \geq 1}^{} m_{r + 1 + t} \big(  (\alpha^{s-1})^{\otimes r} \otimes m_s \otimes (\alpha^{s-1})^{\otimes t}  \big) = 0.
\end{align}
\end{defn}

\medskip

Like classical case, an $HA_\infty$-structure is related to an $HA_\infty[1]$-structure by a degree shift. 
Let $(A, \mu_k, \alpha)$ be an $HA_\infty$-algebra. Take $V = s A$, where $V_i = (s A)_i = A_{i-1}$. Define
\begin{align}
\varrho_k = (-1)^{\frac{k(k-1)}{2}}~s \circ \mu_k \circ (s^{-1})^{\otimes k},
\end{align}
where $s : A \rightarrow sA$ is a degree $+1$ map and $s^{-1} : sA \rightarrow A$ is a degree $-1$ map. Note that $\mu_k$ can be reconstructed from $\varrho_k$ as $\mu_k = s^{-1} \circ \varrho_k \circ s^{\otimes k}.$ The sign $(-1)^{\frac{k(k-1)}{2}}$ is a consequence of the Koszul sign convention. See \cite{lada-markl} for more details. Then we have the following.
\begin{prop}\label{ha-ha1-prop}
An $HA_\infty$-algebra structure on a graded vector space $A$ is equivalent to an $HA_\infty [1]$-structure on the graded vector space $V = sA$. 
\end{prop}

\begin{remark}
Note that Definition \ref{ha} is natural in the sense that the differential $\mu_1$ has degree $-1$, the product $\mu_2$ preserves the degree and $\mu_n$'s are higher homotopies, for $n \geq 3$. However, one can see that Definition \ref{ha1} is simpler in sign.
\end{remark}

\begin{defn}
Let $(A, \mu_k, \alpha)$ and $(A', \mu_k', \alpha')$ be two $HA_\infty$-algebras. A morphism $f : A \rightarrow A'$ of $HA_\infty$-algebras consists of maps $f_k : A^{\otimes k} \rightarrow A'$ of degree $k-1$, for $k \geq 1$, such that
\begin{center}
$f_k \big( \alpha (a_1) \otimes \cdots \otimes \alpha (a_k ) \big) = \alpha' \big( f_k (a_1 \otimes \cdots \otimes a_k) \big)$
\end{center}
and for all $n \geq 1$,
\begin{align}\label{ha-map-id}
&\sum_{r+s+t = n, r , t \geq 0, s \geq 1}^{}  (-1)^{rs + t}~  f_{r+1+t} ( (\alpha^{s-1})^{\otimes r} \otimes \mu_s \otimes (\alpha^{s-1})^{\otimes t} ) \\
& \qquad \qquad = \sum_{l=1}^{n} \sum_{i_1 + \cdots + i_l = n, i_1, \ldots, i_l \geq 1}^{} (-1)^{u}~ \mu_l' ( \alpha'^{(i_2 -1)+ \cdots + (i_l -1)} f_{i_1} \otimes \cdots \otimes \alpha'^{(i_1 -1)+ \cdots + (i_{l-1} - 1)} f_{i_l}  ), \nonumber
\end{align}
where $u = (l-1) (i_1 -1) + (l-2) (i_2 -1) + \cdots + 2 (i_{l-2}-1) + (i_{l-1} - 1)$.
\end{defn}

If $f : A \rightarrow A'$ is a $HA_\infty$-algebra morphism, then $f_1 : A \rightarrow A'$ is a morphism of complexes which is compatible with $\mu_2$ up to a homotopy given by $f_2$. Therefore, $f_1$ induces a graded hom-associative algebra morphism at the level of homology.

A morphism $f : A \rightarrow A'$ of $HA_\infty$-algebras is called strict if $f_k = 0$, for all $k \geq 2$. In this case, the identity (\ref{ha-map-id}) becomes
$$f_1 \circ \mu_n = \mu_n' \circ (f_1 \otimes \cdots \otimes f_1).$$
Therefore, strict morphisms of $HA_\infty$-algebras are analogous to hom-associative algebra morphisms. A strict morphism $f : A \rightarrow A'$ is called strict isomorphism if $f_1$ is an isomorphism.

A morphism $f : A \rightarrow A'$ of $HA_\infty$-algebras is called a quasi-isomorphism if $f_1 : A \rightarrow A'$ is a quasi-isomorphism of complexes. In other words, the induced map $H_* (f_1) : H_* (A) \rightarrow H_* (A')$ is an isomorphism of graded hom-associative algebras. In this case, we write $A \simeq A'$. 

We also introduce a notion of morphism between $HA_\infty[1]$-algebras.

\begin{defn}
Let $(V, \varrho_k, \alpha)$ and $(V', \varrho_k', \alpha')$ be two $HA_\infty[1]$-algebras. A morphism $f: V \rightarrow V'$ between them consists of multilinear maps $f_k : V^{\otimes k} \rightarrow V'$ of degree $0$, for $k \geq 1$, such that 
\begin{center}
	$f_k \big( \alpha (v_1) \otimes \cdots \otimes \alpha (v_k ) \big) = \alpha' \big( f_k (v_1 \otimes \cdots \otimes v_k) \big)$
\end{center}
 and for all $n \geq 1$,
\begin{align}
& \sum_{r+s+t = n, r , t \geq 0, s \geq 1}^{} ~  f_{r+1+t} ( (\alpha^{s-1})^{\otimes r} \otimes \varrho_s \otimes (\alpha^{s-1})^{\otimes t} ) \\
& \qquad \qquad = \sum_{l=1}^{n} \sum_{i_1 + \cdots + i_l = n, i_1, \ldots, i_l \geq 0}^{}~ \varrho_l' ( \alpha'^{(i_2 -1)+ \cdots + (i_l -1)} f_{i_1} \otimes \cdots \otimes \alpha'^{(i_1 -1)+ \cdots + (i_{l-1} - 1)} f_{i_l} ). \nonumber
\end{align}
\end{defn}

Let $f : V \rightarrow V'$ and $g : V' \rightarrow V''$ be two $HA_\infty[1]$-algebra morphism. Then their composition $g \circ f : V \rightarrow V''$ is given by
$$(g \circ f)_k = \sum_{i_1 + \cdots + i_l = k, i_1, \ldots, i_l \geq 1}^{}  g_l \circ (  f_{i_1} \otimes \cdots \otimes f_{i_l}).$$

\subsection{Coderivation interpretation}\label{ha-coder-subsec}

Let $V$ be a graded vector space. Consider the graded vector space $TV = \oplus_{k \geq 0} V^{\otimes k}$ with the coproduct $\triangle : TV \rightarrow (TV) \otimes (TV)$ defined by
\begin{center}
$\triangle (v_1 \otimes \cdots \otimes v_n) = 
 \sum_{i=0}^{n} ( v_1 \otimes \cdots \otimes  v_i) \otimes ( v_{i+1} \otimes \cdots \otimes v_n).$
\end{center}
Note that, if $\alpha : V \rightarrow V$ be a linear map of degree $0$, then $\alpha$ induces a coalgebra map
(denoted by the same symbol) $\alpha : TV \rightarrow TV$ defined by
\begin{center}
$\alpha (v_1 \otimes \cdots \otimes v_n) = \alpha (v_1) \otimes \cdots \otimes \alpha (v_n).$
\end{center}

\medskip

The following lemma is very useful to derive an equivalent description of $HA_\infty$-algebras.

\begin{lemma}\label{lifting-lemma}
 Let $V$ be a graded vector space and $\alpha : V \rightarrow V$ be a linear map of degree $0$. Let $\varrho : V^{\otimes n} \rightarrow V$ be a linear map $($of degree $|\varrho|)$ satisfying $\alpha \circ \varrho = \varrho \circ \alpha^{\otimes n},$ that is,
\begin{center}
	$ \alpha ( \varrho (v_1 \otimes \cdots \otimes v_n)) = \varrho (   \alpha (v_1) \otimes \cdots \otimes \alpha (v_n) ).$
	\end{center}
	View $\varrho$ as a linear map $\varrho : TV \rightarrow V$ by letting its only non-zero component being given by the original $\varrho$ on $V^{\otimes n}$. Then $\varrho$ lifts uniquely to a map $\widetilde{\varrho} : TV \rightarrow TV$ $($of degree $|\varrho|)$ satisfying $\alpha \circ \widetilde{\varrho} = \widetilde{\varrho} \circ \alpha$ and
\begin{align}\label{coder-lemma}
	\triangle \circ \widetilde{\varrho} = (\widetilde{\varrho} \otimes \alpha^{n-1} + \alpha^{n-1} \otimes \widetilde{\varrho}) \circ \triangle.
\end{align}
	More precisely, $\widetilde{\rho}$ is given by\\
$\widetilde{\varrho} (v_1 \otimes \cdots \otimes v_k) :=$
$$\begin{cases}0, ~~~ & \mbox{ if } k < n,\\
\sum_{i=0}^{k-n} (-1)^{|\varrho| (|v_1|+ \cdots + |v_i|)}~ \alpha^{n-1}v_1 \otimes \cdots \otimes \alpha^{n-1} v_i \otimes \varrho (v_{i+1} \otimes \cdots \otimes v_{i+n}) \otimes \cdots \otimes \alpha^{n-1} v_k, & \mbox{ if } k \geq n.
\end{cases} $$
Thus, $\widetilde{\varrho}$ maps $V^{\otimes k} \rightarrow V^{\otimes k -n +1}$, for $k \geq n$.	
\end{lemma}

\begin{proof}
	It is easy to verify that $\widetilde{\varrho}$ satisfies the above properties. Next, we prove the uniqueness. Let $\overline{\rho} : TV \rightarrow TV$ be another such map satisfying the above properties. Let $\overline{\rho}^j$ denote the component of $\overline{\rho}$ mapping $ {(TV)} \rightarrow V^{\otimes j}.$ Since $\overline{\rho}$ satisfies (\ref{coder-lemma}), we have
\begin{align*}
\triangle (\overline{\rho} (v_1 \otimes \cdots \otimes v_k)) =~& (\overline{\rho} \otimes \alpha^{n-1} + \alpha^{n-1} \otimes \overline{\rho}) \triangle (v_1 \otimes \cdots \otimes v_k) \\
=~&  (\overline{\rho} \otimes \alpha^{n-1} + \alpha^{n-1} \otimes \overline{\rho}) \big\{ \sum_{i=0}^{k} (v_1 \otimes \cdots \otimes v_i) \otimes ( v_{i+1} \otimes \cdots \otimes v_k) \big\} \\
=~& \sum_{i=0}^{k} \big\{ \overline{\rho} (v_1 \otimes \cdots \otimes v_i )  \otimes (\alpha^{n-1} v_{i+1} \otimes \cdots \otimes \alpha^{n-1} v_k)  \\
&\quad  + (-1)^{|\overline{\rho}| (|v_1|+ \cdots + |v_i|)} (\alpha^{n-1}v_1 \otimes \cdots \otimes \alpha^{n-1} v_i) \otimes \overline{\rho} ( v_{i+1} \otimes \cdots \otimes v_k) \big\}.
\end{align*}
By projecting both sides to $V^{\otimes i} \otimes V^{\otimes j} \subset~ {(TV)} \otimes {(TV)}$ with $i+j = m$, we get
\begin{align*}
\triangle (\overline{\rho}^m (v_1 \otimes \cdots \otimes v_k))|_{V^{\otimes i} \otimes V^{\otimes j}} =~& \overline{\rho}^i ( v_1 \otimes \cdots \otimes v_{k-j}) \otimes (\alpha^{n-1}v_{k-j+1} \otimes \cdots \otimes \alpha^{n-1}v_k) \\
+& (-1)^{|\overline{\rho}| (|v_1|+ \cdots + |v_i|)} (\alpha^{n-1}v_1 \otimes \cdots \otimes \alpha^{n-1} v_i) \otimes \overline{\rho}^j ( v_{i+1}\otimes \cdots \otimes v_k).
\end{align*}

We have $\overline{\varrho}^1 = \varrho$ from the assumption.
Hence, for $m=i=1$ and $j=0$, we get $\overline{\rho}^0 = 0$. For $m=2$, choose $i=1$ and $j=1$, we get that $\overline{\varrho}^2$ is non-zero on $V^{\otimes n+1}$ and is given by 
\begin{center}
$\overline{\varrho}^2 (v_1 \otimes \cdots \otimes v_{n+1})= \varrho (v_1 \otimes \cdots \otimes v_n) \otimes \alpha^{n-1} v_{n+1} ~+~ (-1)^{|\varrho||v_1|} \alpha^{n-1} v_1 \otimes \varrho (v_2 \otimes \cdots \otimes v_{n+1}).$
\end{center}
We use induction to determine $\overline{\varrho}^m$, for any $m \geq 2$. More precisely, for any $m \geq 2$, we choose $i = m-1$ and $j = 1$ to determine $\overline{\varrho}^m$ uniquely by using $\overline{\varrho}^{m-1}$ and $\overline{\varrho}^1$. It can be checked that $\overline{\varrho}^m$ is non-zero only on $V^{\otimes n+m-1}$ and is given by
\begin{align*}
\overline{\varrho}^m &(v_1 \otimes \cdots \otimes v_{n+m-1}) \\&= \sum_{i=0}^{m-1} (-1)^{|\varrho| (|v_1| + \cdots + |v_i|)} \alpha^{n-1} v_1 \otimes \cdots \otimes \alpha^{n-1} v_i \otimes \varrho (v_{i+1} \otimes \cdots \otimes v_{i+n}) \otimes \cdots \otimes \alpha^{n-1} v_{n+m-1}.
\end{align*}
Hence, the proof.
\end{proof}

Let
\begin{center}
$\text{Coder}^p_\alpha (TV) = \{ D = \sum_{n \geq 1} \widetilde{\varrho_n} |~ \varrho_n : V^{\otimes n} \rightarrow V \text{ is a map of degree } p \text{ satisfying } \alpha \circ \varrho_n = \varrho_n \circ \alpha^{\otimes n} \}.$
\end{center}
We first define a graded Lie algebra structure on $\text{Coder}^\bullet_\alpha (TV)$. Let $D = \sum_{n \geq 1} \widetilde{\varrho_n} \in \text{Coder}^p_\alpha (TV)$ and $D' = \sum_{m \geq 1} \widetilde{\sigma_m} \in \text{Coder}^q_\alpha (TV)$. Then each $\varrho_n : V^{\otimes n} \rightarrow V$ is a map of degree $p$ and each $\sigma_m : V^{\otimes m} \rightarrow V$ is a map of degree $q$. We define
$$D \circ D' = \sum_{n \geq 1} \widetilde{\varrho_n} \circ \sum_{m \geq 1} \widetilde{\sigma_m} = \sum_{n, m \geq 1} \widetilde{\varrho_n} \circ \widetilde{\sigma_m} \in \text{Coder}^{p+q}_\alpha (TV).$$
It can be checked that the product $\circ$ defines a graded pre-Lie algebra structure on $\text{Coder}^\bullet_\alpha (TV).$ Hence, the graded commutator
\begin{center}
$[D,D'] = D \circ D' - (-1)^{pq} D' \circ D,$
\end{center}
for $D \in \text{Coder}^p_\alpha (TV)$ and $D' \in \text{Coder}^q_\alpha (TV)$, defines a graded Lie algebra structure on $\text{Coder}^\bullet_\alpha (TV).$

\begin{prop}\label{coder-coder-zero-equiv}
	Let $V$ be a graded vector space and $\alpha : V \rightarrow V$ be a linear map of degree $0$. Let $D = \sum_{n \geq 1} \widetilde{\varrho_n} \in \text{Coder}^{-1}_{\alpha} (TV)$. Then the condition $D \circ D = 0$ is equivalent to the following equations:
	\begin{align*}
	& \varrho_1 (\varrho_1 (v_1)) = 0,\\
	& \varrho_1 ( \varrho_2 (v_1, v_2)) + \varrho_2 (  \varrho_1 (v_1), v_2) + (-1)^{|v_1|} \varrho_2 (v_1, \varrho_1 (v_2)) = 0,\\
	& \cdots \\
	& \sum_{i+j = n+1}^{} \sum_{\lambda=1}^{j} (-1)^{|v_1| + \cdots + |v_{\lambda -1}|}~ \varrho_j \big( \alpha^{i-1 }v_1, \ldots, \alpha^{i-1}v_{\lambda -1}, \varrho_i (v_\lambda, \ldots, v_{\lambda + i-1}), \ldots, \alpha^{i-1} v_n  \big) = 0,\\
	& \cdots.
	\end{align*}
\end{prop}

\begin{proof}
Note that, the condition $D \circ D = 0$ is same as
\begin{center}
$ (\widetilde{\varrho_1} + \widetilde{\varrho_2}+ \widetilde{\varrho_3}+ \cdots ) \circ (\widetilde{\varrho_1} + \widetilde{\varrho_2}+ \widetilde{\varrho_3}+ \cdots ) = 0.$
\end{center}
This is equivalent to the conditions that
\begin{align*}
&\widetilde{\varrho_1} \circ \widetilde{\varrho_1} = 0 ,\\
&\widetilde{\varrho_1} \circ \widetilde{\varrho_2} + \widetilde{\varrho_2} \circ \widetilde{\varrho_1} = 0 , \\
&\cdots\\
&\widetilde{\varrho_1} \circ \widetilde{\varrho_n} + \widetilde{\varrho_2} \circ \widetilde{\varrho_{n-1}} + \cdots + \widetilde{\varrho_{n-1}} \circ \widetilde{\varrho_2} + \widetilde{\varrho_n} \circ \widetilde{\varrho_1} = 0,\\
&\cdots .
\end{align*}
Hence, by applying the definition of $\widetilde{\varrho_i}$, for $i=1, \ldots,$ we get the conclusion.
\end{proof}

In view of Proposition \ref{ha-ha1-prop} and Proposition \ref{coder-coder-zero-equiv}, we have the following result.
\begin{thm}\label{coder-ha-inf}
An	$HA_\infty$-algebra structure on a graded vector space $A$ with respect to a linear map $\alpha : A \rightarrow A$ of degree $0$ is equivalent to an element $D \in Coder^{-1}_\alpha (TV)$ with $D \circ D = 0$, where $V = sA$.
\end{thm}

\subsection{Homotopy transfer theorems} Let $(A, d_A)$ and $(H, d_H)$ be two chain complexes. A contraction from $A$ to $H$ consists of chain maps $p : A \rightarrow H$ and $i : H \rightarrow A$ and a degree $+ 1$ chain homotopy $h : A \rightarrow A$ 
\[
\xymatrix{
h \lefttorightarrow	(A, d_A) \ar@<2pt>[r]^{\quad p} & (H, d_H) \ar@<2pt>[l]^{\quad i}
}
\]
satisfying
\begin{align*}
1_A - i \circ p =~& d_A h + h d_A,\\
p \circ i =~& 1_H.
\end{align*}

Now, suppose $(A, \mu, \alpha, d_A)$ is a differential graded hom-associative algebra. 
We also assume that $\alpha$ commute with the chain homotopy $h$, that is, $\alpha \circ h = h \circ \alpha$.
We aim to prove that the differential graded hom-associative algebra structure on $(A, d_A)$ transfer to an $HA_\infty$-algebra structure on $(H, d_H)$. A similar result for differential graded associative algebras has been proved by Kadeisvili \cite{kad}. We define a multiplication $\mu_2^H : H^{\otimes 2} \rightarrow H$ by 
\begin{center}
$ \mu_2^H (x,y) := p \mu ( i(x), i(y)),~ \text{ for } x, y \in H$
\end{center}
and a degree $0$ linear map $\alpha_H : H \rightarrow H$ by $\alpha_H := p \circ \alpha \circ i$. 
If the map $i$ is an isomorphism with inverse $p$, then the multiplication on $H$ becomes hom-associative.
In general, there is no reason for $\mu_2^H$ to be hom-associative.

The next lemma is straightforward but a tedious calculation.

\begin{lemma}\label{htt-lemma}
We have $\partial (\mu_2^H) = 0$. Moreover, for any $x, y, z \in H$,
\begin{center}
$ \mu_2^H ( \mu_2^H (x,y), \alpha_H (z)  ) - \mu_2^H ( \alpha_H (x), \mu_2^H (y, z)  )  = \partial (\mu_3^H)(x, y, z),$
\end{center}
where $\mu_3^H : H^{\otimes 3} \rightarrow H$ is given by
\begin{center}
$\mu_3^H (x, y, z) =~ p \mu \big( h \mu ( i(x), i(y)) ,~ \alpha i (z) \big) - p \mu \big(  \alpha i (x), ~h \mu ( i(y), i(z))   \big).$
\end{center}
\end{lemma}

The above lemma shows that the obstruction of the hom-associativity of $\mu_2^H$ is given by $\partial (\mu_3^H)$. 
The Remark \ref{rem-ha-derivation} now suggests that $(H, d_H)$ may have an $HA_\infty$-algebra structure.
Like classical case of $A_\infty$-algebras, one can construct maps $\mu_k^H : H^{\otimes k} \rightarrow H$ of degree $k-2$, for $k \geq 3$, satisfying the properties (\ref{ha-derivation-defn}). For $k=3$, it coincides with $\mu_3^H$ defined in the above lemma. Therefore, it will follows that $(H, d_H)$ inherits an $HA_\infty$-algebra structure.

Our method to construct maps $\mu_k^H$, for $k \geq 3$, are similar to the classical case. We follow the method as in \cite{markl}. More precisely, we construct maps $\mu_k^A : A^{\otimes k} \rightarrow A$ of degree $k-2$, for $k \geq 3$, such that
\begin{center}
$\mu_k^H := p \circ \mu_k^A \circ i^{\otimes k}, ~~~~ k \geq 3.$
\end{center}
We define maps $\mu_k^A : A^{\otimes k} \rightarrow A$ by using induction. We set $\mu_2^A = \mu$ and inductively define
$$ \mu_k^A := \sum_{i+j=k, i, j \geq 1}^{} (-1)^{i(j+1)} \mu (\alpha^{j-1} \circ h \circ \mu_i^A \otimes \alpha^{i-1} \circ h \circ \mu_j^A)$$
with the convention that $h \circ \mu_1^A =$ id. We remark that the maps $\mu_k^A$ can also be define in terms of planar binary trees with $k$ leaves, for $k \geq 3$. Note that there is a unique planar binary tree with two leaves, namely,
\begin{tikzpicture}[scale=0.1]
\draw (0,0) -- (0,-2); \draw (0,0) -- (2,2); \draw (0,0) -- (-2,2);
\end{tikzpicture}
.
It correspond to the binary operation $\mu_2^A = \mu$. For any planar binary tree $T$ with $k$ leaves, it is the grafting of two planar binary trees each of which has less than $k$ leaves. Let $T = T_1 \vee T_2$, where $T_1$ is a planar binary tree with $i$ leaves and $T_2$ is a planar binary tree with $j$ leaves with $i+j = k$. Here $\vee$ stands for grafting of trees. For such $T$, we associate a map $\mu_T : A^{\otimes k} \rightarrow A$ by
\begin{center}
$\mu_T  = (-1)^{i(j+1)}~ \mu ( \alpha^{j-1} \circ h \circ \mu_{T_1} \otimes \alpha^{i-1} \circ h \circ \mu_{T_2})$
\end{center}
with the convention that $h \circ \mu_1$= id, where $1$ is the unique planar tree $|$ with $1$ leaf. Finally, define
$$\mu_k^A := \sum_{T \in PBT(k)}^{} ~\mu_T ,$$
where $PBT(k)$ is the set of isomorphism classes of planar binary trees with $k$ leaves.

\medskip

Thus, we get the following $HA_\infty$-algebra version of the homotopy transfer theorem. The proof is similar to Markl \cite{markl} for $A_\infty$-algebras (see also \cite{kad}).

\begin{thm}\label{htt-1}
Let	
\[
\xymatrix{
h \lefttorightarrow	(A, d_A) \ar@<2pt>[r]^{\quad p} & (H, d_H) \ar@<2pt>[l]^{\quad i}
}
\]	
be a contraction data. If $(A, d_A)$ is a differential graded hom-associative algebra, then $(H, d_H)$ inherits a structure of an $HA_\infty$-algebra.
\end{thm}

In the next, we prove that if $(A, d_A)$ inherits a structure of an $HA_\infty$-algebra, this structure also transfers to $(H, d_H)$. A similar for $A_\infty$-algebras has been proved by Kontsevich and Soibelman \cite{kont-soib}. Note that, in Lemma \ref{htt-lemma} we have used the fact that $\mu$ is associative. If $(A, d_A)$ is an $HA_\infty$-algebra, the multiplication $\mu_2$ on $A$ is only associative up to homotopy, that is, there exists a ternary operation $\mu_3$ on $A$ such that
\begin{center}
$\partial (\mu_3) =~  \mu_2 (\mu_2 \otimes \alpha) - \mu_2 (\alpha \otimes \mu_2).$
\end{center}
As before, we construct maps $\mu_k^A : A^{\otimes k} \rightarrow A$ of degree $k-2$, for $k \geq 3$ such that $\mu_k^H := p \circ \mu_k^A \circ i^{\otimes k},$ $k \geq 3.$
Define
$$ \mu_k^A = \sum_{i_1 + \cdots + i_l = k, i_1, \ldots, i_l \geq 1, 2 \leq l \leq k}^{} (-1)^{\theta (i_1, \ldots, i_l)}~ \mu_l (\alpha^{\sum_{p \neq 1} (i_p -1)} \circ h \circ \mu^A_{i_1} \otimes \cdots \otimes \alpha^{\sum_{p \neq l} (i_p -1)} \circ h \circ \mu^A_{i_l}),$$
where $\theta (i_1, \ldots, i_l) = \sum_{1 \leq s < t \leq l} i_s (i_t + 1).$ As before, we have used the convention that $h \circ \mu_1^A =$ id. One may also describe these maps in terms of planar trees (instead of planar binary trees). The proof of the next theorem is similar to Markl \cite{markl}.

\begin{thm}\label{htt-2}
Given a contraction data as in Theorem \ref{htt-1}, if $(A, d_A)$ is an $HA_\infty$-algebra, then $(H, d_H)$ inherits a structure of an $HA_\infty$-algebra.
\end{thm}

\section{$2$-term $HA_\infty$-algebras}\label{sec4}
An $n$-term $HA_\infty$-algebra is an $HA_\infty$-algebra $(A, \mu_k, \alpha)$ whose underlying graded vector space $A$ is concentrated in degrees $0$ to $n-1$, namely,
\begin{center}
$A = A_0 \oplus A_1 \oplus \cdots \oplus A_{n-1}.$
\end{center}
In this case, it is easy to see that $\mu_k = 0$, for $k > n+1.$ In the next section (see Theorem \ref{cat-are-equiv}) we show that
$2$-term $HA_\infty$-algebras are related to categorification of hom-associative algebras. Therefore, we prefer to rewrite the explicit description of $2$-term $HA_\infty$-algebras.

\begin{defn}\label{defn-2ha-inf}
	A $2$-term $HA_\infty$-algebra consists of the following data
	\begin{itemize}
		\item a complex of vector spaces $A_1 \xrightarrow{d} A_0$
		\item bilinear maps $\mu_2 : A_i \otimes A_j \rightarrow A_{i+j}$,
		\item a trilinear map $\mu_3 : A_0 \otimes A_0 \otimes A_0 \rightarrow A_1$,
		\item two linear maps $\alpha_0 : A_0 \rightarrow A_0$ and $\alpha_1 : A_1 \rightarrow A_1$ satisfying $\alpha_0 \circ d = d \circ \alpha_1$ and
		\begin{center}
		$ \alpha_1 \circ \mu_3 = \mu_3 \circ \alpha_0^{\otimes 3}$
		\end{center}
	\end{itemize}
	such that for any $a, b, c, d \in A_0$ and $m, n \in A_1$, the following equalities are satisfied:
	\begin{itemize}
		\item[(a)] $\mu_2 (m,n) = 0$,
		\item[(b)] $\alpha_0 \big(  \mu_2 (a,b)    \big) =  \mu_2 \big(   \alpha_0 (a), \alpha_0 (b) \big)$,
		\item[(c)] $\alpha_1  \big(  \mu_2 (a,m) \big) = \mu_2 \big(  \alpha_0 (a), \alpha_1 (m) \big),$
	    \item[(d)] $\alpha_1  \big(  \mu_2 (m,a) \big) = \mu_2  ( \alpha_1 (m), \alpha_0 (a) \big),$
		\item[(e)] $d \mu_2 (a,m) = \mu_2 (a, dm)$,
		\item[(f)] $d \mu_2 (m,a) = \mu_2 (dm, a)$,
		\item[(g)] $\mu_2 (dm, n) = \mu_2 (m, dn)$, 
		\item[(h)] $d\mu_3 (a, b, c) = \mu_2 \big(   \mu_2 (a,b), \alpha_0 (c) \big) - \mu_2 \big(   \alpha_0 (a), \mu_2 (b,c) \big)$,
		\item[(i1)] $\mu_3 (a, b, dm) =  \mu_2 \big(  \mu_2 (a,b), \alpha_1 (m) \big) - \mu_2 \big(  \alpha_0 (a), \mu_2 (b,m) \big),$
		\item[(i2)] 	$\mu_3 (a, dm, c) =  \mu_2 \big(  \mu_2 (a,m), \alpha_0 (c) \big) - \mu_2 \big(  \alpha_0 (a), \mu_2 (m,c) \big) $,
		\item[(i3)]  $\mu_3 (dm, b, c) =  \mu_2 \big(  \mu_2 (m,b), \alpha_0 (c) \big) - \mu_2 \big(  \alpha_1 (m), \mu_2 (b,c) \big) $,
    	\item[(j)] $\mu_3  \big( \mu_2 (a,b), \alpha_0 (c), \alpha_0(d) \big) - \mu_3 \big(  \alpha_0(a), \mu_2 (b,c), \alpha_0 (d) \big) + \mu_3 \big(    \alpha_0 (a), \alpha_0 (b), \mu_2 (c,d) \big) \\
		 = \mu_2 \big(   \mu_3 (a,b,c), \alpha_0^2 (d) \big) + \mu_2 \big(  \alpha_0^2 (a), \mu_3 (b,c,d) \big).$ 
	\end{itemize}
\end{defn}

We denote a $2$-term $HA_\infty$-algebra as above by  $(A_1 \xrightarrow{d} A_0, \mu_2, \mu_3, \alpha_0, \alpha_1)$.
One may also write the explicit description of morphism between $2$-term $HA_\infty$-algebras. 

\begin{defn}
Let $A= (A_1 \xrightarrow{d} A_0, \mu_2, \mu_3, \alpha_0, \alpha_1)$ and $A'= (A'_1 \xrightarrow{d'} A'_0, \mu'_2, \mu'_3, \alpha'_0, \alpha'_1)$ be two $2$-term $HA_\infty$-algebras. A morphism $f : A \rightarrow A'$ consists of
\begin{itemize}
\item a chain map $f : A \rightarrow A'$ (which consists of linear maps $f_0 : A_0 \rightarrow A_0'$ and $f_1 : A_1 \rightarrow A_1'$ with $f_0 \circ d = d' \circ f_1$) satisfying
\begin{center}
$\alpha_0' \circ f_0 = f_0 \circ \alpha_0~~~ \text{ and } ~~~ \alpha_1' \circ f_1 = f_1 \circ \alpha_1,$
\end{center}
 \item a bilinear map $f_2 : A_0 \otimes A_0 \rightarrow A_1'$ satisfying $\alpha_1' \circ f_2 = f_2 \circ \alpha_0^{\otimes 2}$
\end{itemize}
such that for any $a, b, c \in A_0$ and $m \in A_1$, the followings are hold
\begin{itemize}
\item[(a)] $d' f_2 (a,b) = f_0 (\mu_2 (a,b)) - \mu_2' (f_0 (a), f_0 (b)),$
\item[(b)] $f_2 (a, dm) = f_1 (\mu_2 (a,m)) - \mu_2' (f_0 (a), f_1(m)),$
\item[(c)] $f_2 (dm, a) = f_1 (\mu_2 (m, a)) - \mu_2' (f_1 (m), f_0 (a))$,
\item[(d)] $f_2 (\mu_2 (a,b), \alpha_0 (c)) - f_2 ( \alpha_0 (a), \mu_2 (b, c)) - \mu_2' (f_2 (a,b), \alpha_0' f_0 (c)) + \mu_2' (\alpha_0' f_0 (a), f_2 (b,c)) \\
= f_1 (\mu_3 (a,b, c)) - \mu_3' (f_0 (a), f_0 (b), f_0 (c)).$
\end{itemize}
\end{defn}

If $f = (f_0, f_1, f_2) : A \rightarrow A'$ and $g = (g_0, g_1, g_2) : A' \rightarrow A''$ be two morphism of $2$-term $HA_\infty$-algebras, their composition $g \circ f : A \rightarrow A''$ is defined by $(g \circ f)_0 = g_0 \circ f_0, ~ (g \circ f)_1 = g_1 \circ f_1$ and
\begin{center}
$(g \circ f)_2 (a, b) = g_2 (f_0 (a), f_0 (b)) + g_1 (f_2 (a,b)), ~~ \text{ for } a, b \in A_0.$
\end{center}
For any $2$-term $HA_\infty$-algebra $A$, the identity morphism $\text{id}_A : A \rightarrow A$ is given by the identity chain map $A \rightarrow A$ together with $(1_A)_2 = 0$.

\begin{prop}
	The collection of $2$-term $HA_\infty$-algebras and morphisms between them form a category. We denote this category by ${\bf 2HA_\infty}$.
\end{prop}

\subsection{Skeletal $2$-term $HA_\infty$-algebras}\label{subsec-skeletal}
\begin{defn}
	A $2$-term $HA_\infty$-algebra $(A_1 \xrightarrow{d} A_0, \mu_2, \mu_3, \alpha_0, \alpha_1)$ is called skeletal if $d = 0$. 
\end{defn}

Let $(A_1 \xrightarrow{0} A_0, \mu_2, \mu_3, \alpha_0, \alpha_1)$ be a skeletal $2$-term $HA_\infty$-algebra. Then it follows from conditions (b) and (h) in Definition 3.6 that $(A_0, \mu_2, \alpha_0)$ forms a hom-associative algebra. Moreover, it follows from (i1), (i2) and (i3) that the maps
\begin{align*}
\mu_2 : A_0 \otimes A_1 \rightarrow A_1, ~~ (a,m) \mapsto \mu_2 (a,m) \\
\mu_2 : A_1 \otimes A_0 \rightarrow A_1, ~~ (m,a) \mapsto \mu_2 (m, a)
\end{align*}
defines a hom-bimodule structure on the vector space $A_1$ with respect to the linear map $\alpha_1 : A_1 \rightarrow A_1$.

\begin{defn}\label{equiv-defn}
	Let $(A_1 \xrightarrow{0} A_0, \mu_2, \mu_3, \alpha_0, \alpha_1)$ and 	$(A_1 \xrightarrow{0} A_0, \mu_2', \mu_3', \alpha_0, \alpha_1)$ be  
two skeletal $2$-term $HA_\infty$-algebras on the same chain complex and maps $\alpha_0, \alpha_1$. They are said to be equivalent if $\mu_2 = \mu_2'$ and there exists a linear map $\sigma : A_0 \otimes A_0 \rightarrow A_1$  with $\alpha_1 (\sigma (a,b)) = \sigma (\alpha_0 (a), \alpha_0 (b))$ satisfying
\begin{align*}
\mu_3' (a,b,c) = \mu_3 (a,b,c) ~+~& \mu_2 (\alpha_0 (a), \sigma (b,c)) -~ \sigma (\mu_2 (a,b), \alpha_0 (c)) \\
&+ \sigma (\alpha_0 (a), \mu_2 (b,c)) - \mu_2 (\sigma (a,b), \alpha_0 (c)), \text{   for all } a, b , c \in A_0.
\end{align*}
\end{defn}

\begin{thm}\label{skeletal-2}
	There is a one-to-one correspondence between skeletal $2$-term $HA_\infty$-algebras and tuples $( (A, \mu, \alpha), M, \beta , \cdot, \theta)$, where $(A, \mu, \alpha)$ is a hom-associative algebra, $(M, \beta, \cdot)$ is a hom-bimodule over $A$ and $\theta$ is a $3$-cocycle of the hom-associative algebra $A$ with coefficients in the hom-bimodule.
	
Moreover, it extends to a one-to-one correspondence between equivalence classes of skeletal $2$-term $HA_\infty$-algebras and the third Hochschild cohomology $H^3_{\alpha, \beta} (A, M)$.
\end{thm}

\begin{proof}
	Let $(A_1 \xrightarrow{0} A_0, \mu_2, \mu_3, \alpha_0, \alpha_1)$ be a skeletal $2$-term $HA_\infty$-algebra. We observed that $(A_0, \mu_2, \alpha_0)$ is a hom-associative algebra and the maps $\mu_2 : A_0 \otimes A_1 \rightarrow A_1, ~ \mu_2 : A_1 \otimes A_0 \rightarrow A_1$ defines a hom-bimodule structure on the vector space $A_1$ with respect to the linear map $\alpha_1 : A_1 \rightarrow A_1$. It remains to find out a $3$-cocycle of the hom-associative algebra $A_0$ with coefficients in the hom-bimodule $A_1$. We show that
	\begin{center}
	$ \mu_3 : A_0 \otimes A_0 \otimes A_0 \rightarrow A_1$
	\end{center}
	defines a cocycle. Note that condition (j) is same as
	\begin{align*}
	\mu_2 \big( \alpha_0^2 (a), \mu_3 (b, c, d)  \big) - \mu_3 \big(  \mu_2 (a,b), \alpha_0 (c), \alpha_0 (d)    \big) + \mu_3 \big( \alpha_0(a), \mu_2 (b,c), \alpha_0 (d)    \big) \\ - \mu_3 \big(  \alpha_0(a), \alpha_0(b), \mu_2 (c,d) \big) + \mu_2 \big( \mu_3 (a,b,c), \alpha_0^2 (d)    \big) = 0.
	\end{align*}
	This is equivalent to $\delta_{\alpha_0, \alpha_1} (\mu_3) (a, b, c, d) = 0$.
	
	Conversely, given a tuple $((A, \mu, \alpha), M, \beta, \cdot, \theta)$ as in the statement, define $A_0 = A$, $A_1= M$, $\alpha_0 = \alpha$ and $\alpha_1 = \beta$. We define multiplications $\mu_2 : A_i \otimes A_j \rightarrow A_{i+j}$ and $\mu_3 : A_0 \otimes A_0 \otimes A_0 \rightarrow A_{1}$ by
	\begin{align*}
	\mu_2 (a,b) =~& \mu (a,b),\\
	\mu_2 (a,m) =~& a \cdot m,\\
	\mu_2 (m,a) =~& m \cdot a,\\
	\mu_3 (a,b,c) =~& \theta (a,b,c),
	\end{align*}
	for  $a, b, c \in A_0 = A$ and $m \in A_1 = M.$ Then it is easy to verify that $(A_1 \xrightarrow{0} A_0, \mu_2, \mu_3, \alpha_0, \alpha_1)$ is a skeletal $2$-term $HA_\infty$-algebra.
	
	Finally, let $(A_1 \xrightarrow{0} A_0, \mu_2, \mu_3, \alpha_0, \alpha_1)$ and
	$(A_1 \xrightarrow{0} A_0, \mu_2, \mu_3', \alpha_0, \alpha_1)$ be two equivalent skeletal algebras with the equivalence given by $\sigma$ (cf. Definition \ref{equiv-defn}). Then it follows that
	\begin{center}
$ \mu_3' = \mu_3 + \delta_{\alpha_0, \alpha_1} (\sigma)$
\end{center}
as a $3$-cochain in $C^3_{\alpha_0, \alpha_1} (A, M)$. Hence, we have $[\mu_3] = [\mu_3'].$ Conversely, any two representatives of a same cohomology class in $H^3_{\alpha_0, \alpha_1} (A, M)$ gives rise to equivalent skeletal algebras.
\end{proof}

More generally, we can consider $n$-term $HA_\infty$-algebras with the underlying chain complex
\begin{center}
$ A = ( A_{n-1} \xrightarrow{0} 0 \xrightarrow{0} \cdots \xrightarrow{ 0} 0 \xrightarrow{0} A_0 )$
\end{center}
consists of only two non-zero terms $A_0$ and $A_{n-1}$ with the zero differential. The linear map $\alpha : A \rightarrow A$ of degree $0$ consists of two components $\alpha_0 : A_0 \rightarrow A_0$ and $\alpha_{n-1} = A_{n-1} \rightarrow A_{n-1}$. Moreover, the multiplications $\mu_k$'s are non-zero only for $k=2, n+1$. Two such $n$-term $HA_\infty$-algebras
$( A_{n-1} \xrightarrow{0} 0 \xrightarrow{0} \cdots \xrightarrow{ 0} 0 \xrightarrow{0} A_0 , \mu_2, \mu_{n+1}, \alpha_0, \alpha_{n-1})$ and $( A_{n-1} \xrightarrow{0} 0 \xrightarrow{0} \cdots \xrightarrow{ 0} 0 \xrightarrow{0} A_0 , \mu_2', \mu_{n+1}', \alpha_0, \alpha_{n-1})$ on the same chain complex and linear map $\alpha$ are said to be equivalent if $\mu_2 = \mu_2'$ and
there exists a linear map $\sigma : (A_0)^{\otimes n} \rightarrow A_{n-1}$ satisfying $\alpha_{n-1} \circ \sigma = \sigma \circ \alpha_0^{\otimes n}$ and such that
\begin{align*}
\mu_{n+1}' (a_1, \ldots, a_{n+1}) = ~&\mu_{n+1} (a_1, \ldots, a_{n+1}) + \mu_2 ( \alpha_0^{n-1} (a_1), \sigma (a_2, \ldots, a_{n+1})) \\
&+ \sum_{i=1}^{n} (-1)^i ~\sigma \big(\alpha_0 (a_1), \ldots, \alpha_0 (a_{i-1}) , \mu_2 (a_i, a_{i+1}), \alpha_0 (a_{i+2}), \ldots, \alpha_0 (a_{n+1}) \big) \\
&+ (-1)^{n+1} \mu_2 (\sigma (a_1, \ldots, a_{n}), \alpha_0^{n-1} (a_{n+1})),
\end{align*}
for all $a_1, \ldots, a_{n+1} \in A_0.$

We have the following.

\begin{thm}\label{skeletal-n}
	There is a one-to-one correspondence between $n$-term $HA_\infty$-algebras consisting of only two non-zero terms $A_0$ and $A_{n-1}$, with $d= 0$, and tuples $((A, \mu, \alpha), M, \beta, \cdot, \theta)$, where $(A, \mu, \alpha)$ is a hom-associative algebra, $(M, \beta, \cdot)$ is a hom-bimodule over $A$ and $\theta$ is a $(n+1)$-cocycle of $A$ with coefficients in the hom-bimodule.
	
	Moreover, it extends to a one-to-one correspondence between equivalence classes of $n$-term $HA_\infty$-algebras of the above type and the $(n+1)$-th Hochschild cohomology $H^{n+1}_{\alpha, \beta} (A, M)$.
\end{thm}

\begin{proof}
Let $( A_{n-1} \xrightarrow{0} 0 \xrightarrow{0} \cdots \xrightarrow{ 0} 0 \xrightarrow{0} A_0 , \mu_2, \mu_{n+1}, \alpha_0, \alpha_{n-1})$ be a $n$-term $HA_\infty$-algebra of the given type.
It follows from the definition of $HA_\infty$-algebra that $A = (A_0, \mu_2, \alpha_0)$ is a hom-associative algebra. Moreover, the maps $\mu_2 : A_0 \otimes A_{n-1} \rightarrow A_{n-1}$ and $\mu_2 : A_{n-1} \otimes A_0 \rightarrow A_{n-1}$ defines a hom-bimodule structure on the vector space $M= A_{n-1}$ with respect to the linear map $\beta = \alpha_{n-1} : M \rightarrow M$. Next, we claim that the map
\begin{center}
	$\mu_{n+1} : (A_0)^{\otimes n+1} \rightarrow A_{n-1}$
\end{center}
	defines a $(n+1)$-cocycle. Note that the condition (\ref{ha-eqn}) in the definition of $HA_\infty$-algebra implies that
	\begin{align*}
	\sum_{i+j = n+3}^{} \sum_{\lambda =1}^{j} (-1)^{\lambda (i+1) + i (|a_1| + \cdots + |a_{\lambda -1 }|)} ~ \mu_{j} \big(  \alpha_0^{i-1}a_1, \ldots, \alpha_0^{i-1} a_{\lambda - 1}, \mu_i ( a_{\lambda}, \ldots, a_{\lambda + i-1}), \ldots, \alpha_0^{i-1} a_{n+2}   \big) = 0,
	\end{align*}
	for all $a_1, \ldots, a_{n+2} \in A_0$. 
The non-zero terms in the above summation occurs for $(i=2,~ j = n+1)$ and $(i = n+1,~ j = 2)$. Therefore, we get
	\begin{align*}
	\sum_{\lambda = 1}^{n+1} & (-1)^{3 \lambda} ~\mu_{n+1}  \big( \alpha_0(a_1), \ldots, \alpha_0(a_{\lambda -1}) , \mu_2 (a_\lambda , a_{\lambda +1} ), \alpha_0 (a_{\lambda +2}), \ldots, \alpha_0 (a_{n+2})   \big) \\
	&+ (-1)^{n+2} \mu_2  \big(    \mu_{n+1} (a_1, \ldots, a_{n+1}), \alpha_0^n (a_{n+2}) \big ) + \mu_2 \big( \alpha_0^n (a_1), \mu_{n+1} (a_2, \ldots, a_{n+2})   \big) = 0,
	\end{align*}
	for all $a_1, \ldots, a_{n+2} \in A_0$.
	This is equivalent to the fact that $\mu_{n+1} : (A_0)^{\otimes n+1} \rightarrow A_{n-1}$ defines a $(n+1)$-cocycle of the hom-associative algebra $(A_0, \mu_2, \alpha_0)$ with coefficients in the hom-bimodule $M=A_{n-1}$.
	
	The converse part is similar to Theorem \ref{skeletal-2}.
	
	The correspondence between the equivalence classes of $n$-term $HA_\infty$-algebras of given type and the $(n+1)$-th Hochschild cohomology is also similar to the last part of Theorem \ref{skeletal-2}.
\end{proof}

\subsection{Strict $2$-term $HA_\infty$-algebras}\label{subsec-strict}

\begin{defn}
	A $2$-term $HA_\infty$-algebra $(A_1 \xrightarrow{d} A_0, \mu_2, \mu_3, \alpha_0, \alpha_1)$ is called strict if $\mu_3 = 0$.
\end{defn}

\begin{exam}\label{exam-strict}
Let $(A, \mu, \alpha)$ be a hom-associative algebra. Take $A_0 = A_1 = A$, $d =~$id,~$\mu_2 = \mu$ and $\alpha_0 = \alpha_1 = \alpha$. Then $(A_1 \xrightarrow{d} A_0, \mu_2, \mu_3 = 0, \alpha_0, \alpha_1)$ is a strict $2$-term $HA_\infty$-algebra.
\end{exam}

Now we introduce crossed module of hom-associative algebras and show that strict $2$-term $HA_\infty$-algebras are in one-to-one correspondence with crossed module of hom-associative algebras.
\begin{defn}
	A crossed module of hom-associative algebras consist of a tuple $((A, \mu_A, \alpha_A), (B, \mu_B, \alpha_B), dt, \phi)$ where $(A, \mu_A, \alpha_A)$ and $(B, \mu_B, \alpha_B)$ are hom-associative algebras, $dt : A \rightarrow B$ is a hom-associative algebra morphism and
	\begin{align*}
	\phi : B \otimes A \rightarrow A,& ~ (b, m) \mapsto \phi (b ,m), \\
	\phi : A \otimes B \rightarrow A,& ~ (m, b) \mapsto \phi(m, b)
	\end{align*}
	defines a hom-bimodule structure on $A$ with respect to the linear map $\alpha_A$, such that for $ m, n \in A,~ b \in B$,
	\begin{align*}
	dt ( \phi (b, m)) =~& \mu_B (b, dt(m)),\\
	dt ( \phi(m, b)) = ~& \mu_B (dt (m), b  ), \\
	\phi ({dt(m)} , n) =~& \mu_A (m, n),\\
	\phi(m, dt(n)) =~& \mu_A (m, n),\\
	\phi (  \alpha_B (b), \mu_A (m, n)) =~& \mu_A (\phi(b, m), \alpha_A (n)),\\
	\phi (  \mu_A (m,n), \alpha_B (b)) =~& \mu_A (\alpha_A (m), \phi(n,b)).
	\end{align*}
\end{defn}

\begin{thm}\label{strict-crossed-mod}
	There is a one-to-one correspondence between strict $2$-term $HA_\infty$-algebras and crossed module of hom-associative algebras.
\end{thm}
\begin{proof}
	Let $(A_1 \xrightarrow{d} A_0, \mu_2 , \mu_3 = 0, \alpha_0, \alpha_1)$ be a strict  $2$-term $HA_\infty$-algebra. Take $B = A_0$ with the multiplication $\mu_B = \mu_2 : A_0 \otimes A_0 \rightarrow A_0$ and the linear transformation $\alpha_B = \alpha_0$. It follows from (b) and (h) that $(B, \mu_B, \alpha_B)$ is a hom-associative algebra. Take $A = A_1$ with the multiplication $\mu_A$ is given by $\mu_A (m,n) := \mu_2 (dm, n) = \mu_2 (m, dn)$, for $m, n \in A_1$, and the linear transformation $\alpha_A = \alpha_1$. By condition (c) we have
	\begin{center}
	$\alpha_A (   \mu_A (m,n)) = \alpha_1 ( \mu_2 (dm,n)) =~ \mu_2 \big( \alpha_0 (dm), \alpha_1 (n) \big)
	=~ \mu_2 \big( d (\alpha_1m) , \alpha_1(n) \big) 
	=~ \mu_A \big( \alpha_1 (m), \alpha_1 (n)  \big),$
	\end{center}
	which implies that $\alpha_A$ is an algebra morphism with respect to $\mu_A$. Moreover,
	\begin{align*}
	\mu_A & (  \alpha_A (m), \mu_A (n,p)  ) - \mu_A (   \mu_A (m,n), \alpha_A (p)) \\
	=~& \mu_2 (   d\alpha_1 (m), \mu_2 (dn, p) ) - \mu_2 ( d \mu_2 (dm,n) , \alpha_1 (p)  ) \\
	=~& \mu_2 (   \alpha_0 (dm), \mu_2 (dn,p) )  -  \mu_2 (  \mu_2 (dm, dn), \alpha_1 (p)) \quad(\text{by  (e)}) \\
	=~& 0 \quad (\text{by  (i1)} ).
	\end{align*}
	This shows that $(A, \mu_A, \alpha_A)$ is a hom-associative algebra.
	
	Take $dt = d : A_1 \rightarrow A_0$. It commute with the linear transformations. Moreover, we have
	\begin{center}
	$d ( \mu_A (m,n) ) = d ( \mu_2 (dm,n)) = \mu_2 (dm,dn) = \mu_B (dm, dn), ~~\text{ for } m, n \in A_1 = A.$
	\end{center}
	Hence, $dt$ is a morphism of hom-associative algebras. Finally, we define
	\begin{align*}
	\phi : B \otimes A \rightarrow A, ~ (b, m) & \mapsto  \phi (b,m) = \mu_2 (b, m), \\
	\phi : A \otimes B \rightarrow A, ~ (m, b) & \mapsto \phi(m,b) = \mu_2 (m , b),  \text{ for } b \in B, m \in A.
	\end{align*}
	By using (i1), (i2) and (i3), it is easy to see that $\phi$ defines a hom-bimodule structure on $A$ with respect to the linear map $\alpha_A$. Moreover,
	\begin{center}
$	dt (  \phi (b, m)) = d ( \mu_2 (b,m)) = \mu_2 (b, dm),$    ~~~  \text{   ~~~(by (e))}
\end{center}
\begin{center}
$	dt (  \phi (m, b)) = d (\mu_2 (m,b)) = \mu_2 (dm, b),$   ~~~  \text{   ~~~(by (f))}
\end{center}
\begin{center}
$	\phi ({dt (m)}, n) = \mu_2 (  dm, n) = \mu_A (m,n),$
	\end{center}
	\begin{center}
$	\phi ({m}, dt(n)) = \mu_2 (m, dn) = \mu_A (m,n),$
	\end{center}
\begin{align*}
	 \phi ( \alpha_B (a), \mu_A (m,n)) = \mu_2 (  \alpha_0 (a), \mu_2 (dm, n) ) = \mu_2 ( \mu_2 (a,dm), \alpha_1 (n)) = \mu_A (\phi (a,m), \alpha_A (n) ), \\  ~~~ \text { ~~~~  (by (i1) and (e))}
\end{align*}
and
\begin{align*}
\phi ( \phi_A (m,n) , \alpha_B (a) ) = \mu_2 (   \mu_2 (m,dn), \alpha_0 (a) ) =   \mu_2 ( \alpha_1 (m),  \mu_2 (dn, a) ) = \mu_A (\alpha_A (m), \phi (n,a)), \\ ~~~~ \text{  (by (i3) and (f))}.
\end{align*}
	Therefore, $(A, B, dt, \phi)$ is a crossed module of hom-associative algebras.
	
\medskip
	
Conversely, given a crossed module of hom-associative algebras $((A, \mu_A, \alpha_A), (B, \mu_B, \alpha_B), dt, \phi)$, we can construct a strict $2$-term $HA_\infty$-algebra as follows. Take $A_1 = A$, $A_0 = B$ and $dt : A \rightarrow B$ is the chain map $d: A_1 \rightarrow A_0$. The structure map $\mu_2 : A_i \otimes A_j \rightarrow A_{i+j}$ is given by
	\begin{center}
	$\mu_2 (b, b') = \mu_B (b, b'), ~~ \mu_2 (b, m) = \phi (b, m),~~ \mu_2 (m, b) = \phi (m, b) \text{~~ and ~~} \mu_2 (m,n) = 0,$
	\end{center}
	for $b, b' \in A_0 = B,$ and $m, n \in A_1 = A$. The homomorphisms $\alpha_1 : A_1 \rightarrow A_1$ and $\alpha_0 : A_0 \rightarrow A_0$ are given by $\alpha_1 = \alpha_A$ and $\alpha_0 = \alpha_B$. It is easy to verify that these structures gives rise to a strict $2$-term $HA_\infty$-algebra structure $(A_1 \xrightarrow{d} A_0, \mu_2, \mu_3 = 0, \alpha_0, \alpha_1)$.
\end{proof}

Note that the crossed module corresponding to the strict $2$-term $HA_\infty$-algebra of Example \ref{exam-strict} is given by $((A, \mu, \alpha), (A, \mu, \alpha), \text{id}, \mu)$.

\section{Hom-associative $2$-algebras}\label{sec5}

Let {\bf Vect} denote the category of vector spaces over $\mathbb{R}$. A $2$-vector space is a category internal to the category {\bf Vect}. Thus, a $2$-vector space $C$ is a category with a vector space of objects $C_0$ and a vector space of arrows (morphisms) $C_1$, such that all the structure maps are linear. Let $s, t : C_1 \rightarrow C_0$ be the source and target map respectively, $i : C_0 \rightarrow C_1$ be the identity assigning map. A morphism of $2$-vector spaces is a functor internal to the category {\bf Vect}. More precisely, a morphism (linear functor) $F : C \rightarrow C'$ between $2$-vector spaces consist of linear maps $F_0 : C_0 \rightarrow C_0'$ and $F_1 : C_1 \rightarrow C_1'$ commute with structure maps. We denote the category of $2$-vector spaces by {\bf $2$-Vect}. Given a $2$-vector space $C = (C_1 \rightrightarrows C_0)$, we have a $2$-term complex
\begin{align}\label{2vect-2term}
\text{ker}~(s) \xrightarrow{t} C_0.
\end{align}
A morphism $F : C \rightarrow C'$ of $2$-vector spaces induce a morphism between corresponding $2$-term chain complexes.
Conversely, any $2$-term complex $A_1 \xrightarrow{d} A_0$ gives rise to a $2$-vector space
\begin{align}\label{2term-2vect}
\mathbb{V} = (A_0 \oplus A_1 \rightrightarrows A_0)
\end{align}
 of which the set of objects is $A_0$ and the set of morphisms is $A_0 \oplus A_1$. The structure maps are given by $s (a,m) = a$, $t(a,m) = a + dm$ and $i(a)= (a, 0)$. A morphism between $2$-term chain complexes induce a morphism between corresponding $2$-vector spaces.

 We denote the category of $2$-term complexes of real vector spaces by {\bf $2$-Term}. Then there is an equivalence of categories
\begin{center}
{\bf $2$-Term}  $\simeq$ ~{\bf $2$-Vect}. 
\end{center}
See \cite{baez-crans} for more details. In \cite{baez-crans} the authors also introduced Lie $2$-algebras as categorification of Lie algebras. They showed that the category of $2$-term $L_\infty$-algebras and the category
of Lie $2$-algebras are equivalent. In \cite{sheng-chen} the authors proved a similar result in the context of hom-Lie algebras. More precisely, they introduced $HL_\infty$-algebras as well as hom-Lie $2$-algebras and showed that the category of $2$-term $HL_\infty$-algebras and the category of hom-Lie $2$-algebras are equivalent.

In the previous section, we have studied $HA_\infty$-algebra structures on a $2$-term complex. In this section, we introduce hom-associative $2$-algebra structures on a $2$-vector space. Finally, we show that the category of $2$-term $HA_\infty$-algebras and the category of hom-associative $2$-algebras are equivalent.

\begin{defn}\label{hom-ass-2-alg-defn}
A hom-associative $2$-algebra is a $2$-vector space $C$ equipped with
\begin{itemize}
	\item a bilinear functor $\mu : C \otimes C \rightarrow C$,
	\item a linear functor $\Phi = (\Phi_0, \Phi_1) : C \rightarrow C$ satisfying
	\begin{center}
	$ \Phi (  \mu (\xi, \eta)  ) = \mu ( \Phi(\xi), \Phi(\eta ) ), ~ \text{ for all } \xi, \eta \in C,$
	\end{center}
	\item a trilinear natural isomorphism, called the hom-associator
	\begin{center}
	$\mathcal{A}_{\xi, \eta, \zeta} : \mu ( \mu(\xi, \eta), \Phi (\zeta)  ) \rightarrow   \mu   (   \Phi (\xi), \mu (\eta, \zeta)) $
	\end{center}
\end{itemize}
satisfying
\begin{center}
$\mathcal{A}_{\Phi_0 (\xi), \Phi_0 (\eta), \Phi_0 (\zeta)} = \Phi_1 \mathcal{A}_{\xi, \eta, \zeta}$
\end{center}
and such that the following identity is satisfied
\[
\xymatrixrowsep{0.5in}
\xymatrixcolsep{0.1in}
\xymatrix{
	&	\mu \big(~ \mu ( \mu(\xi , \eta) , \Phi_0 (\zeta) ) ,~ \Phi_0^2 (\lambda) \big) \ar[ld]_{\mathcal{A}_{\xi, \eta, \zeta}} \ar[rd]^{\mathcal{A}_{\mu(\xi, \eta), \Phi_0 (\zeta), \Phi_0 (\lambda)}} &  \\
\mu \big(~ \mu (\Phi_0 (\xi) , \mu (\eta , \zeta)) ,~ \Phi_0^2 (\lambda) \big) \ar[d]_{\mathcal{A}_{\Phi_0(\xi), \mu (\eta, \zeta), \Phi_0 (\lambda)}} & & \mu \big(~	\Phi_0 (\mu(\xi , \eta)) , ~ \mu  (\Phi_0 (\zeta) , \Phi_0 (\lambda) ) ~\big) \ar[d]^{\text{id}} \\
\mu \big(   \Phi_0^2 (\xi) , ~\mu (  \mu (\eta , \zeta) , \Phi_0 (\lambda)) ~\big) \ar[rd]_{\mathcal{A}_{\eta, \zeta, \lambda}} & & \mu \big(~ \mu (\Phi_0 (\xi) , \Phi_0 (\eta)) , ~  \Phi_0 (\mu(\zeta , \lambda))~ \big) \ar[ld]^{~~\mathcal{A}_{\Phi_0 (\xi), \Phi_0 (\eta), \mu(\zeta, \lambda)}} \\
 & \mu \big( \Phi_0^2 (\xi) , ~ \mu ( \Phi_0 (\eta) , \mu (\zeta , \lambda) ) ~\big) &  \\
}
\]
\end{defn}

\medskip

A hom-associative $2$-algebra as above may be denoted by $(C, \mu, \Phi, \mathcal{A})$. A hom-associative $2$-algebra is called strict if the hom-associator is the identity isomorphism.
\begin{defn}
	Let $(C, \mu, \Phi, \mathcal{A})$ and $(C', \mu', \Phi', \mathcal{A}')$ be two hom-associative $2$-algebra. A hom-associative $2$-algebra morphism consists of
	\begin{itemize}
		\item a linear functor $(F_0, F_1)$ from the underlying $2$-vector space $C$ to $C'$ such that
\begin{center}	
	$\Phi' \circ (F_0, F_1) = (F_0, F_1) \circ \Phi,$
\end{center}
		\item a bilinear natural transformation
		\begin{center}
		$ F_2 (\xi, \eta) : \mu' (F_0 (\xi), F_0 (\eta))  \rightarrow F_0 ( \mu (\xi, \eta))$
		\end{center}
	\end{itemize}
		satisfying $F_2 ( \Phi_0 (\xi), \Phi_0 (\eta)) = \Phi_1' (F_2 (\xi, \eta))$ and such that the following diagram commute
\[
\xymatrixrowsep{0.5in}
\xymatrixcolsep{0.01in}
\xymatrix{
 &	\mu' \big(~ \mu' ( F_0 (\xi) , F_0 (\eta) ) , ~ \Phi' (F_0 (\zeta))~ \big) \ar[rd]^{F_2 (\xi, \eta)} \ar[ld]_{\mathcal{A}'_{F_0 (\xi), F_0 (\eta), F_0 (\zeta)}} &  \\
\mu' \big(~  \Phi' (F_0 (\xi)) , ~\mu' (   F_0(\eta) , F_0 (\zeta))~ \big) \ar[d]_{F_2 (\eta, \zeta)} & &  \mu' \big(~ F_0 (\mu(\xi , \eta)) ,~ F_0 (\Phi(\zeta)) \big)  \ar[d]^{F_2 \big(\mu(\xi, \eta), \Phi (\zeta) \big)} \\
\mu' \big( ~ F_0 (\Phi(\xi)) ,~ F_0 (\mu(\eta , \zeta)) ~\big) \ar[rd]_{F_2 \big(\Phi (\xi), \mu (\eta, \zeta) \big)} & &  F_0 \big(~ \mu( \mu (\xi , \eta) , \Phi (\zeta) ) ~\big) \ar[ld]^{\mathcal{A}_{\xi, \eta, \zeta}} \\
  & F_0 \big(~ \mu ( \Phi (\xi) , \mu(\eta , \zeta) ) ~\big) & \\
}
\]
\end{defn}

\medskip

The composition of two hom-associative $2$-algebra morphisms is again a hom-associative $2$-algebra morphism. More precisely, let $C, C'$ and $C''$ be three hom-associative $2$-algebras and $F : C \rightarrow C'$, $G: C' \rightarrow C''$ be hom-associative $2$-algebra morphisms. Their composition $G \circ F : C \rightarrow C''$ is a hom-associative $2$-algebra morphism whose components are given by
$ (G \circ F)_0 = G_0 \circ F_0 , ~ (G \circ F)_1 = G_1 \circ F_1$
and $(G \circ F)_2$ is given by
\[
\xymatrix{
\mu'' ( G_0 \circ F_0 (\xi), G_0 \circ F_0 (\eta))  \ar[rd]_{G_2 (F_0 (\xi), F_0 (\eta) )} \ar[rr]^{(G \circ F)_2 (\xi, \eta)} & & (G_0 \circ F_0) (\mu (\xi, \eta)) \\
& G_0 ( ~ \mu'   ( F_0 (\xi) , F_0 (\eta) ) ~ ) \ar[ru]_{G_0 (F_2 (\xi, \eta))} &
}
\]

For any hom-associative $2$-algebra $C$, the identity morphism $\text{id}_C : C \rightarrow C$ is given by the identity functor as its linear functor together with the identity natural transformation as $(\text{id}_C)_2$.

\medskip

Therefore, we obtain the following result.
\begin{prop}
Hom-associative $2$-algebras and hom-associative $2$-algebra morphisms form a category. We denote this category by {\bf HAss2}.
\end{prop}

We now prove the main result of this section.

\begin{thm}\label{cat-are-equiv}
	The categories ${\bf 2HA\infty}$ and {\bf HAss2} are equivalent.
\end{thm}

\begin{proof}
First we construct a functor $T : {\bf 2HA_\infty}  \rightarrow {\bf HAss2}$ as follows. Given a $2$-term $HA_\infty$-algebra $A = ( A_1 \xrightarrow{d} A_0, \mu_2, \mu_3, \alpha_0, \alpha_1)$, we have the $2$-vector space
$C = (A_0 \oplus A_1 \rightrightarrows A_0)$ as defined by (\ref{2term-2vect}). Define a bilinear functor $\mu : C \otimes C \rightarrow C$ by
\begin{center}
$ \mu \big(   (a,m), (b,n)  \big)  = \big(  \mu_2 (a,b), \mu_2 (a,n) + \mu_2 (m,b) + \mu_2 (dm,n)  \big),$
\end{center}
for $(a,m), (b,n) \in C_1 = A_0 \oplus A_1$. The linear functor $\Phi : C \rightarrow C$ is given by $\Phi = (\Phi_0, \Phi_1) := (\alpha_0, \alpha_1)$.
Moreover, We have
\begin{align*}
\Phi \big(  \mu  (   (a,m), (b,n)  ) \big) =~& \big(    \alpha_0 \mu_2 (a,b) ,~ \alpha_1   (\mu_2 (a,n) + \mu_2 (m,b) + \mu_2 (dm,n))            \big) \\
=~& \big(    \mu_2 (\alpha_0 a, \alpha_0 b), ~ \mu_2 (\alpha_0 a, \alpha_1 n) + \mu_2 (\alpha_1 m, \alpha_0 b)  + \mu_2 ( d \alpha_0 m, \alpha_1 n)     \big) \\
=~& \mu \big( (\alpha_0 a, \alpha_1 m) , (\alpha_0 b, \alpha_1 n)   \big) = \mu \big( \Phi (a,m) , \Phi (b,n)   \big).
\end{align*}
Define the hom-associator by
\begin{center}
$\mathcal{A}_{a,b,c} = \big( \mu (\mu(a,b), \alpha_0 (c)) ,~ \mu_3 (a,b,c)    \big).$
\end{center}
Note that
\begin{align*}
\mathcal{A}_{\Phi_0 (a), \Phi_0 (b), \Phi_0 (c)} =~&  \big( \mu (\mu( \alpha_0 a, \alpha_0 b), \alpha^2_0 (c)) ,~ \mu_3 (\alpha_0 a, \alpha_0 b, \alpha_0 c)    \big) \\
=~&  \big( \alpha_0 \mu (\mu(a,b), \alpha_0 (c)) ,~ \alpha_1 \mu_3 (a,b,c)    \big) 
= \Phi_1 \mathcal{A}_{a,b,c}.
\end{align*}
By using identities (a)-(j), one can also verify that the diagram in the Definition \ref{hom-ass-2-alg-defn} commutes.
Therefore, $(C, \mu, \mathcal{A}, \Phi)$ is a hom-associative $2$-algebra.

For any $HA_\infty$-morphism $f = (f_0, f_1, f_2)$ from $A$ to $A'$, we define a morphism $F$ from $C = T(A)$ to $C' = T(A')$ as follows.  Take $F_0 = f_0$, $F_1 = f_1$ and
\begin{center}
$ F_2 (a,b) = \big(  \mu' ( f_0 (a), f_0 (b) ) ,~ f_2 (a,b)   \big).$
\end{center}
It is easy to verify that $F$ is a morphism from $C$ to $C'$. Moreover, one can verify that $T$ preserves the identity morphisms and the composition of morphisms. Therefore, $T$ is a functor from ${\bf 2HA_\infty}$ to ${\bf HAss2}$.

In the next, we construct a functor $S : {\bf HAss2} \rightarrow {\bf 2HA_\infty}$ as follows. Given a hom-associative $2$-algebra $C = (C_1 \rightrightarrows C_0, \mu, \Phi, \mathcal{A})$, we have the $2$-term chain complex
\begin{center}
$ A_1 = \text{ker } s \xrightarrow{d=t|_{\text{ker} s}} C_0 = A_0$
\end{center}
as defined by (\ref{2vect-2term}). Define $\mu_2 : A_i \otimes A_j \rightarrow A_{i+j}$ by
$$ \mu_2 (a,b) = \mu (a,b), ~~~ \mu_2 (a,m) = \mu (i(a), m), ~~~ \mu_2 (m, a) = \mu (m, i(a)) ~~~\text{ and } ~~~ \mu_2 (m,n) = 0.$$
The map $\mu_3 : A_0 \otimes A_0 \otimes A_0 \rightarrow A_1$ is defined by
\begin{center}
$ \mu_3 (a,b,c) = \mathcal{A}_{a,b,c} - i (    s (\mathcal{A}_{a,b,c})), ~~ \text{ for } a, b, c \in A_0.$
\end{center}
Define the homomorphism $\alpha_0$ and $\alpha_1$ by $\alpha_0 = \Phi_0 : A_0 (= L_0) \rightarrow A_0$ and $\alpha_1 = \Phi_1|_{A_1 = \text{ker } s } : V_1 \rightarrow V_1$. Since $\Phi$ is a functor, we have $\alpha_0 \circ d = d \circ \alpha_1$. Note that, since $\mathcal{A}_{\Phi_0 (a), \Phi_0 (b), \Phi_0 (c)} = \Phi_1 \mathcal{A}_{a,b,c}$, we have $\alpha_1 \mu_3 (a,b,c) = \mu_3 (\alpha_0 (a), \alpha_0 (b), \alpha_0 (c))$. Then it is easy to verify that the tuple $(A_1 \xrightarrow{d} A_0, \mu_2, \mu_3 , \alpha_0 , \alpha_1)$ defines a $2$-term $HA_\infty$-algebra.

For any hom-associative $2$-algebra morphism $F = (F_0, F_1, F_2) : C \rightarrow C'$, we define a $HA_\infty$-morphism $S(F) = f = (f_0, f_1, f_2)$ from $S(C)$ to $S(C')$ as follows. Take $f_0 = F_0$, $f_1 = F_1 |_{V_1 = \text{ker } s}$ and define $f_2$ by
$$ f_2 (a,b) = F_2 (a,b) - i (  s F_2 (a,b)).$$
It is easy to verify that $f$ is a $HA_\infty$-morphism. Moreover, $S$ preserves the identity morphism and the composition of morphisms. Therefore, $S$ is a functor from ${\bf HAss2}$ to ${\bf 2HA_\infty}$.

It remains to prove that the composite functor $T \circ S$ is naturally isomorphic to the identity functor $1_{{\bf HAss2}}$, and the composite $S \circ T$ is naturally isomorphic to $1_{{\bf 2HA_\infty}}$. Given a hom-associative $2$-algebra $C = (C, \mu, \Phi, \mathcal{A})$, the corresponding $2$-term $HA_\infty$-algebra $S(C)$ is given by
$$ (A_1 = \text{ker } s \xrightarrow{d = t|_{ker~ s}} A_0, \mu_2, \mu_3, \alpha_0, \alpha_1).$$
Therefore, the hom-associative $2$-algebra $T (S(C)) := (C', \mu', \Phi', \mathcal{A}')$ is given by $C' := (A_0 \oplus A_1 \rightrightarrows A_0)$. Define $\theta : T \circ S \rightarrow 1_{{\bf HAss2}}$ by
$$\theta_C : T \circ S (C) = C' \rightarrow C = 1_{{\bf HAss2}} (C)$$
with $(\theta_C)_0 (a)= a$, $(\theta_C)_1 (a,m)= i(a) + m$. It is obvious that $\theta_C$ is an isomorphism of $2$-vector spaces. 

Moreover, we have
\begin{align*}
\theta_C \big(   \mu' ((a,m), (b,n)) \big) =~&  \theta_C \big(  \mu_2 (a,b),~ \mu_2 (a,n) + \mu_2 (m,b) + \mu_2 (dm,n) \big) \\
=~& \theta_C \big( \mu (a,b) ,~ \mu ( i(a), n) +  \mu (m, i(b) ) + \mu (i (dm) , n )      \big) \\
=~& i ( \mu(a,b)  ) +  \mu ( i(a), n) +  \mu (m, i(b) ) + \mu (i(dm), n )\\
=~& \mu (  i(a), i(b))+  \mu ( i(a), n) +  \mu (m, i(b) ) + \mu (m,n) \\
=~& \mu ( i(a) + m,~ i(b) + n ) \\
=~& \mu \big( \theta_C (a,m),~ \theta_C (b,n) \big) .
\end{align*}
This shows that $\theta_C$ is a hom-associative $2$-algebra morphism. Thus, $\theta_C$ is an isomorphism of hom-associative $2$-algebras. It is easy to see that it is a natural isomorphism.

For any $2$-term $HA_\infty$-algebra $A = (A_1 \xrightarrow{d} A_0, \mu_2, \mu_3, \alpha_0, \alpha_1)$, the corresponding hom-associative $2$-algebra $T(A)$ is given by 
$$T(A) = (A_0 \oplus A_1 \rightrightarrows A_0, \mu, \Phi, \mathcal{A}).$$
If we apply the functor $S$ to $T(A)$, we get back the same $2$-term $HA_\infty$-algebra $A$. Therefore, the natural isomorphism $\vartheta : S \circ T \rightarrow 1_{{\bf 2HA_\infty}}$ is given by the identity. Hence, the proof.
\end{proof}

The above theorem can be strengthen in the world of $2$-category. 
More precisely, one can define a notion of $2$-morphism between two morphisms in ${\bf 2HA_\infty}$. In fact, the whole data makes ${\bf 2HA_\infty}$ into a strict $2$-category. Similarly, one can make {\bf HAss2} into a strict $2$-category. Finally, one can prove that the $2$-categories ${\bf 2HA_\infty}$ and {\bf HAss2} are $2$-equivalent. A similar equivalence between the $2$-categories of `$2$-term $L_\infty$-algebras' and `Lie $2$-algebras' has been proved in \cite{baez-crans}.

\section{Cohomology and deformations of $HA_\infty$-algebras}\label{sec6}
In this section, we define Hochschild cohomology of $HA_\infty$-algebras. This cohomology controls the formal deformation of the structure.
\subsection{Cohomology}
 We fix an $HA_\infty$-algebra $A = (A, \mu_k, \alpha)$.
\begin{defn}
 A left $HA_\infty$-module over $A$ consists of a graded vector space $M$ together with
\begin{itemize}
	\item a collection $\{ \eta_k |~ 1 \leq k < \infty \}$ of multilinear maps $\eta_k : \otimes^{k-1} A \otimes M \rightarrow M$ with deg $(\eta_k) = k-2 ,$ 
	\item a linear map $\beta : M \rightarrow M$ of degree $0$ with
	$$ \beta \big( \eta_k (a_1, \ldots, a_{k-1}, m) \big) = \eta_k \big(  \alpha(a_1), \ldots, \alpha(a_{k-1}), \beta (m) \big)$$
	\end{itemize}
	such that for all $n \geq 1$,
	\begin{align}\label{module-rel}
	\sum_{i+j = n+1}^{} \sum_{\lambda = 1}^{j} (-1)^{ \lambda (i+1) + i (|a_1| + \cdots + |a_{\lambda -1}|)} ~ \eta_{j} \big(  \gamma^{i-1}a_1, \ldots, \gamma^{i-1} a_{\lambda -1}, \eta_i ( a_{\lambda}, \ldots, a_{\lambda + i -1}),\\
	\gamma^{i-1} a_{\lambda + i}, \ldots, \gamma^{i-1} a_n   \big) = 0, \nonumber
	\end{align}
	for $a_1, \ldots, a_{n-1} \in A$ and $a_n \in M$.
\end{defn}

Note that in the above definition we use the following conventions. If $\lambda < j$, that is, $\lambda + i -1 < n$,
\begin{align*}
\eta_{j} \big(  \gamma^{i-1} & a_1, \ldots, \gamma^{i-1} a_{\lambda-1}, \eta_i ( a_{\lambda}, \ldots, a_{\lambda + i-1}), \gamma^{i-1} a_{\lambda + i}, \ldots, \gamma^{i-1} a_n   \big) \\
~&= \eta_{j} \big(  \alpha^{i-1}a_1, \ldots, \alpha^{i-1} a_{\lambda-1}, \mu_i ( a_{\lambda}, \ldots, a_{\lambda + i-1}), \alpha^{i-1} a_{\lambda + i}, \ldots, \beta^{i-1} a_n   \big).
\end{align*}
If $\lambda = j$, that is, $a_{\lambda + i-1} = a_n$, we use
\begin{align*}
\eta_{j} \big(  \gamma^{i-1} a_1 &, \ldots, \gamma^{i-1} a_{\lambda-1}, \eta_i ( a_{\lambda}, \ldots, a_{\lambda + i-1}), \gamma^{i-1} a_{\lambda + i}, \ldots, \gamma^{i-1} a_n   \big) \\
~& = \eta_j \big( \alpha^{i-1} a_1, \ldots, \alpha^{i-1} a_{j-1}, \eta_i (a_j, \ldots, a_n ) \big).
\end{align*}

It is easy to observe that $M = A$ is a left $HA_\infty$-module over itself with $\eta_k = \mu_k$ and $\beta = \alpha$. One can also define a right $HA_\infty$-module. The notion of a $HA_\infty$-bimodule is also not hard to define.
For a graded vector space $M$, we denote
$$ J^k_{A, M} = \oplus_{i+j = k-1} \big[ A^{\otimes i} \otimes M \otimes A^{\otimes j} \big].$$

\begin{defn}
An $HA_\infty$-bimodule over $A$ consists of a graded vector space $M$ together with
\begin{itemize}
\item a collection $\{ \eta_k | 1 \leq k < \infty \}$ of multilinear maps $\eta_k : J^k_{A, M} \rightarrow  M$ with deg $(\eta_k) = k-2,$ 
\item a linear map $\beta : M \rightarrow M$ of degree $0$ satisfying
$$ \beta ( \eta_k (a_1, \ldots, a_k) ) = \eta_k ( \gamma (a_1), \ldots, \gamma (a_k)  ) $$
\end{itemize}
and such that for all $n \geq 1$, the identity (\ref{module-rel}) is hold, for all $(a_1, \ldots, a_n) \in J^n_{A, M}$.
\end{defn}

 As before, we have used the following conventions in the above definition:
 
(i) whenever some $a_i \in A$, by $\gamma^{j} (a_i)$, we mean $\alpha^j (a_i)$,

(ii) whenever some $a_i \in M$, by $\gamma^{j} (a_i)$, we mean $\beta^j (a_j),$

(iii) whenever $(a_\lambda, \ldots, a_{\lambda + i -1}) \in A^{\otimes i}$, by $\eta_i (a_\lambda , \ldots, a_{\lambda + i -1})$, we mean $\mu_i (a_\lambda , \ldots, a_{\lambda + i -1}),$

It is easy to see that $A$ is an $HA_\infty$-bimodule over itself. The notion of $HA_\infty[1]$-module and $HA_\infty[1]$-bimodule can be defined in a similar way. They are related to $HA_\infty$-module and $HA_\infty$-bimodule by a degree shift. Next, we introduce Hochschild cohomology of an $HA_\infty$-algebra with coefficients in itself. As a motivation, we first give an equivalent description of Hochschild cohomology of a hom-associative algebra described in subsection \ref{subsec-two-one}.

Let $(A, \mu, \alpha)$ be a hom-associative algebra. It can be considered as a $HA_\infty$-algebra whose underlying graded vector space (denoted by the same symbol $A$) is concentrated in degree $0$. Consider the suspension $V= s A$ which is concentrated in degree $1$. The map $\alpha : A \rightarrow A$ induces a degree $0$ map $\alpha : V \rightarrow V$ and the multiplication $\mu$ on $A$ induces a multiplication (also denoted by the same symbol) $\mu : V^{\otimes 2} \rightarrow V$ of degree $-1$ satisfying
$$\mu (   \alpha (sa) \otimes \alpha (sb) ) = \alpha ( \mu (sa \otimes sb)).$$
Hence, it induces an element $\widetilde{\mu} \in \text{Coder}^{-1}_\alpha (TV)$. The hom-associativity of $\mu$ implies that $\widetilde{\mu} \circ \widetilde{\mu} = 0$. Moreover, we observe that if $f \in C^n_\alpha (A, A)$, it can be considered as a map $f : V^{\otimes n} \rightarrow V$ of degree $- (n-1)$ satisfying
$$f (\alpha (sa_1) \otimes \cdots \otimes \alpha (sa_n)) = \alpha (f (sa_1 \otimes \cdots \otimes sa_n)).$$
Therefore, it lifts to a coderivation $\widetilde{f} \in \text{Coder}^{-(n-1)}_\alpha (TV)$. In fact, we get an isomorphism
$$ \oplus C^\bullet_\alpha (A, A) ~\cong ~ \oplus \text{Coder}^{-(\bullet - 1)}_\alpha (TV), ~ f \mapsto \widetilde{f}.$$
Moreover, for $f \in C^n_\alpha (A, A)$, the Hochschild coboundary is given by
\begin{center}
$\widetilde{\delta_\alpha (f)} = [\widetilde{\mu} , \widetilde{f}] = \widetilde{\mu} \circ \widetilde{f} - (-1)^{(n-1)} \widetilde{f} \circ \widetilde{\mu},$
\end{center}
where on the right hand side, we have used the graded Lie bracket on coderivations.

Let $(A, \mu_k, \alpha)$ be an $HA_\infty$-algebra with the corresponding coderivation is given by $D \in \text{Coder}^{-1}_\alpha (TV)$. The Hochschild cochain complex of $A$ with coefficients in itself is given by $\big(  C^\bullet_\alpha (A, A), \delta_\alpha \big)$, where  $C^\bullet_\alpha (A, A) = \text{Coder}^{-(\bullet -1)} (TV)$ and
$$ \delta_\alpha (D') := [D, D'] = D \circ D' - (-1)^{|D'|} D' \circ D, ~~ \text{for }D' \in \text{Coder}^{|D'|}_\alpha (TV).$$
Since $D \circ D = 0$, it follows that
\begin{align*}
\delta_\alpha^2 (D') =~& [D, D \circ D' - (-1)^{|D'|} D' \circ D]\\
=~& D \circ D \circ D' - (-1)^{|D'|} D \circ D' \circ D - (-1)^{|D'|+1} D \circ D' \circ D + (-1)^{|D'|+ |D'|+1} D' \circ D \circ D = 0.
\end{align*}

One may also define the Hochschild cohomology of an $HA_\infty$-algebra with coefficients in an arbitrary $HA_\infty$-bimodule. See \cite{trad} for the basic idea of the Hochschild cohomology of an $A_\infty$-algebra with coefficients in a bimodule.

It is known that the Hochschild cohomology of an associative algebra carries a Gerstenhaber algebra structure \cite{gers0}. This result has been extended by Getzler and Jones to $A_\infty$-algebras \cite{getz-jon}. Recently, the present author shows that the Hochschild cohomology $H^\bullet_\alpha (A, A)$ of a hom-associative algebra also carries a similar structure \cite{das, das2}. More precisely, given a hom-associative algebra $(A, \mu, \alpha)$, there is a degree $-1$ graded Lie bracket on $C^\bullet_\alpha (A, A)$ given by
\begin{center}
$[f,g] = f \circ g - (-1)^{(p-1)(q-1)} g \circ f, ~~~ \text{ for } f \in C^p_\alpha (A, A), ~ g \in C^q_\alpha (A, A),$
\end{center}
where $(f \circ g)$ is given by
\begin{center}
$(f \circ g)(a_1, \ldots, a_{p+q-1}) = \sum_{i=1}^p (-1)^{(i-1)(q-1)} ~ f (\alpha^{q-1} a_1, \ldots, g (a_i, \ldots, a_{i+q-1}), \ldots, \alpha^{q-1} a_{p+q-1}),$
\end{center}
for $a_1, \ldots, a_{p+q-1} \in A$ (see \cite{amm-ej-makh, das}). There is also a cup-product $\cup_\alpha$ on $C^\bullet_\alpha (A, A)$ defined by
\begin{center}
$(f \cup_\alpha g)(a_1, \ldots, a_{p+q}) = \mu \big( \alpha^{q-1} f (a_1, \ldots, a_m), \alpha^{p-1} g(a_{m+1}, \ldots, a_{m+n}) \big).$
\end{center}
Then the bracket and the cup-product on $C^\bullet_\alpha (A, A)$ induces a Gerstenhaber algebra structure on the cohomology $H^\bullet_\alpha (A, A)$ \cite{das}.

Motivated from the result of \cite{getz-jon} and \cite{das}, we now prove that the Hochschild cohomology of an $HA_\infty$-algebra carries a Gerstenhaber structure. Since the main steps of the proof are similar to the proof of \cite{getz-jon}, we only define the components of the Gerstenhaber structure on the cohomology.

Let $A$ be a graded vector space and $\alpha : A \rightarrow A$ be a linear map of degree $0$. Then $\alpha$ induces a degree $0$ map (denoted by the same symbol) $\alpha : V \rightarrow V$, where $V = sA$. 
We know that the graded space $\text{Coder}^\bullet_\alpha (TV)$ carries a graded Lie algebra structure (see subsection \ref{ha-coder-subsec}). Therefore, the graded vector space $C^\bullet_\alpha (A, A) = \text{Coder}^{-(\bullet - 1)}_\alpha (TV)$ inherits a degree $-1$ graded Lie bracket. More precisely, let
 $D \in C^p_\alpha (A, A)$ and $D' \in C^q_\alpha (A, A)$. Then $D = \sum_{n \geq 1} \widetilde{\varrho_n}$ and $ D' =  \sum_{m \geq 1} \widetilde{\sigma_m}$, where each $\varrho_n : V^{\otimes n} \rightarrow V$ is a  map of degree $-(p-1)$ and each 
$\sigma_m : V^{\otimes m} \rightarrow V$ is a map of degree $-(q-1).$ For $\varrho_n : V^{\otimes n} \rightarrow V$ and $\sigma_m : V^{\otimes m} \rightarrow V$, we define $\varrho_n \circ \sigma_m : V^{n + m -1} \rightarrow V$ by
\begin{align*}
(\varrho_n \circ \sigma_m) (v_1, \ldots, v_{n+m-1}) = \sum_{i=1}^{n} (-1)^{(\sum_{l=1}^{i-1} |v_l|)(q-1)} ~\varrho_n \big(\alpha^{m-1}v_1, \ldots, \sigma_m (v_i, \ldots, v_{i+m-1}), \\
\ldots, \alpha^{m-1} v_{n+m-1} \big).
\end{align*}
Define $D \circ D' = \sum_{k \geq 1} \sum_{n+m-1 = k} \widetilde{\varrho_n \circ \sigma_m}$. The degree $-1$ graded Lie bracket on $C^\bullet_\alpha (A, A)$ is given by
\begin{center}
$ [D , D'] = D \circ D' - (-1)^{(p-1)(q-1)} D' \circ D.$
\end{center}

\medskip

Next, we asuume that $(A, \mu_k, \alpha)$ be an $HA_\infty$-algebra with the corresponding coderivation $D = \sum_{n \geq 1} \widetilde{\varrho_n} \in C^2_\alpha (A, A) = \text{Coder}^{-1}_\alpha (TV)$. For any $D', D'' \in C^\bullet_\alpha (A, A)$, we have
\begin{align*}
 \delta_\alpha ([D',D'']) = [D, [D',D'']]
 =~& [[D, D'], D''] + (-1)^{(q-1)}~ [D', [D, D'']] \\
 =~& [\delta_\alpha (D'), D''] + (-1)^{(q-1)}~ [ D', \delta_\alpha (D'')],
 \end{align*}
for $D' \in C^q_\alpha (A, A)$. This shows that the degree $-1$ graded Lie bracket on $C^\bullet_\alpha (A, A)$ induces a degree $-1$ graded Lie bracket on the cohomology $H^\bullet_\alpha (A, A)$.
We can also define a cup-product $M(-,-)$ on $C^\bullet_\alpha (A, A)$ as follows. For $D' = \sum_{m \geq 1} \widetilde{\sigma_m} \in C^q_\alpha (A, A)$ and $D'' = \sum_{n \geq 1} \widetilde{\tau_n} \in C^r_\alpha (A, A)$, define
$$M(D',D'') = \sum_{l \geq 1} \sum_{m+n+k = l, i \geq 1, i+m \leq j \leq m +k, k \geq 0}  \widetilde{\varrho_{k+2} (\text{id}, \ldots, \text{id}, \sigma_m, \text{id}, \ldots, \text{id}, \tau_n, \text{id}, \ldots, \text{id})},$$
where
\begin{align*}
&\varrho_{k+2} (\text{id}, \ldots, \text{id}, \sigma_m, \text{id}, \ldots, \text{id}, \tau_n, \text{id}, \ldots, \text{id}) (v_1, \ldots, v_{m+n+k}) \\
&= (-1)^{\psi}~ \varrho_{k+2} \big( \alpha^{m+n-2} v_1, \ldots, \alpha^{n-1} \sigma_m (v_i, \ldots, v_{i+m-1}), \ldots, \alpha^{m-1} \tau_n (v_j, \ldots, v_{j+n-1}),\\
& \hspace*{8cm} \ldots, \alpha^{m+n-2} v_{m+n+k} \big)
\end{align*}
with $\psi = (q-1) (\sum_{l=1}^{i-1} |v_l|) ~+~ (r-1) \big( \sum_{l=1}^{j-1} |v_l| \big)$.

Note that, unlike the case of associative algebra or hom-associative algebra, the above defined cup-product is not associative in general. In fact, Getzler and Jones showed that the Hochschild cochains $C^\bullet (A, A)$ of an $A_\infty$-algebra inherits an $A_\infty$-algebra structure \cite{getz-jon}. More precisely, they considered a system of braces $\{-\}\{-, \ldots, -\}$ on $C^\bullet (A, A)$. The $A_\infty$-algebra structure on $C^\bullet (A, A)$ is then given by
\begin{center}
$\psi_k (D_1, \ldots, D_k) = \{D\} \{D_1, \ldots, D_k\}.$
\end{center}
In our case of an $HA_\infty$-algebra, we may also consider braces $\{-\}\{-,\ldots,-\}$ on $C^\bullet_\alpha (A,A)$ suitably twisted by $\alpha$, which will induce an $A_\infty$-algebra structure on $C^\bullet_\alpha (A,A)$. (A detailed description for the braces on a non-graded vector space $A$ together with a linear map $\alpha : A \rightarrow A$ is given in \cite{das2}). Therefore, the cohomology (which is the Hochschild cohomology $H^\bullet_\alpha (A,A)$) carries a graded associative algebra structure. Finally, a similar observation made in Getzler and Jones \cite{getz-jon} (see \cite{das} for the calculation of hom-associative algebra case) one can prove that the induced associative product on the cohomology is graded commutative as well. Moreover, the degree $-1$ graded Lie bracket and the product on the cohomology turns out to satisfy the graded Leibniz rule. Thus, we get the following.

\begin{thm}\label{cohomology-g-alg}
Let $(A, \mu_k, \alpha)$ be a $HA_\infty$-algebra. Then its Hochschild cohomology $H^\bullet_\alpha (A, A)$ carries a Gerstenhaber algebra structure.
\end{thm}

\subsection{Deformations}
Let $(A, \mu_k, \alpha)$ be an $HA_\infty$-algebra with the corresponding coderivation $D \in \text{Coder}^{-1}_\alpha (TV)$, where $V= sA$. A formal one parameter deformation of $A$ is given by a formal sum 
\begin{center}
$D_t = D + t D_1 + t^2 D_2 + \cdots \in \text{Coder}^{-1}_\alpha (TV)[[t]],$
\end{center}
such that $D_t \circ D_t = 0$. This is equivalent to a system of equations
\begin{align*}
&D \circ D = 0,\\
&D \circ D_1 + D_1 \circ D = 0,\\
& \cdots \\
&D \circ D_n + D_1 \circ D_{n-1} + \cdots + D_{n-1} \circ D_1 + D_n \circ D = 0,\\
&\cdots .
\end{align*}
The first equation corresponds to the given $HA_\infty$-algebra structure on $A$. The second equation implies that $\delta_\alpha (D_1) = 0$. This shows that $D_1$ is a $2$-cocycle in the Hochschild cohomology of $A$. The $2$-cocycle $D_1$ is called the infinitesimal of the deformation $D_t$. More generally, if $D_1 = D_2 = \cdots = D_{n-1} = 0$ and $D_n$ is non-zero, then $D_n$ is a $2$-cocycle and is called the $n$-infinitesimal of the deformation $D_t$.

\begin{defn}
Two deformations $D_t = D + \sum_{i \geq 1} t^i D_i$ and $\widetilde{D}_t = D + \sum_{i \geq 1} t^i \widetilde{D}_i$ are said to be equivalent if there is a formal sum
\begin{center}
$ \Phi_t = \Phi_0 + t \Phi_1 + t^2 \Phi_2 + \cdots \in \text{Coder}^0_\alpha (TV)[[t]]$
\end{center}
with $\Phi_0 = \text{id}$ and such that $D_t \circ \Phi_t = \Phi_t \circ \widetilde{D}_t.$
\end{defn}

By equating coeffcients of $t^0$ and $t^1$ in both sides, we obtain
\begin{align*}
 D \circ \Phi_0 =~& \Phi_0 \circ D,\\
 D \circ \Phi_1 + D_1 =~& \widetilde{D}_1 + \Phi_1 \circ D.
\end{align*}
The first equation holds automatically as $\Phi_0 = \text{id}$, while the second equation implies that $D_1 - \widetilde{D}_1 = - \delta_\alpha (\Phi_1).$ This shows that infinitesimals of equivalent deformations are cohomologous. That is, they gives rise to same cohomology class in $H^2_\alpha (A,A).$

\begin{defn}
A deformation $D_t = D + \sum_{i \geq 1} t^i D_i$ of $A$ is called trivial if $D_t$ is equivalent to the deformation $\widetilde{D}_t = D$. An $HA_\infty$-algebra $A$ is called rigid if any deformation of $A$ is trivial.
\end{defn}

\begin{prop}
Let $D_t = D + \sum_{i \geq 1} t^i D_i$ be a deformation of $A$. Then $D_t$ is equivalent to some deformation
\begin{center}
$\widetilde{D}_t = D + t^p \widetilde{D}_p + t^{p+1} \widetilde{D}_{p+1} + \cdots $
\end{center}
in which the first non-vanishing term $\widetilde{D}_p$ is in $Z^2_\alpha (A, A)$ but not in $B^2_\alpha (A, A)$.
\end{prop}
\begin{proof}
Let $D_t = D + \sum_{i \geq 1} t^i D_i$ be a deformation such that $D_1= \cdots = D_{n-1} = 0$ and $D_n$ is the first non-zero term. Then it has been shown that $D_n \in Z^2_\alpha (A, A)$ is a Hochschild $2$-cocycle. If $D_n$ is not in $B^2_\alpha (A, A)$, we are done. If $D_n$ is in $B^2_\alpha (A, A)$, that is, $D_n = - \delta_\alpha (\Phi_n)$, for some $\Phi_n \in C^1_\alpha (A, A)$, then setting
\begin{center}
$ \Phi_t = \text{id} + t^n \Phi_n.$
\end{center}
We define $\overline{D}_t = \Phi_t^{-1} \circ D_t \circ \Phi_t$. Then $\overline{D}_t$ defines a formal deformation of the form
\begin{center}
$\overline{D}_t = D + t^{n+1} \overline{D}_{n+1} + t^{n+2} \overline{D}_{n+2} + \cdots.$
\end{center}
Thus, it follows that $\overline{D}_{n+1}$ is a $2$-cocycle. If this $2$-cocycle is not in $B^2_\alpha (A, A)$, we are done. If this is coboundary, we can apply the same method again. In this way, we get a required type of equivalent deformation.
\end{proof}

As a corollary, we get the following.
\begin{thm}\label{2-zero-rigid}
If $H^2_\alpha (A, A) = 0$, then $A$ is rigid.
\end{thm}

Let $(A, \mu_k, \alpha)$ be a $HA_\infty$-algebra. A deformation of $A$ of order $n$ is of the form $D_t = D + \sum_{i=1}^{n} t^i D_i$ such that $(D_t)^2 = 0$. Next, we shall consider the problem of extending a deformation of order $n$ to a deformation of order $n+1$.
Suppose there is a coderivation $D_{n+1} \in \text{Coder}^{-1}_\alpha (TV)$ such that
$$\widetilde{D}_t = D_t + t^{n+1} D_{n+1}$$
is a deformation of order $n+1$. Then we say that $D_t$ extends to a deformation of order $n+1$.

Since we assume that $D_t = D + \sum_{i = 1}^n t^i D_i$ is a deformation, we have
\begin{align}\label{deform-rel}
D \circ D_i + D_1 \circ D_{i-1} + \cdots + D_{i-1} \circ D_1 + D_i \circ D = 0, ~~~ \text{ for } i = 1, 2, \ldots, n.
\end{align}
Here we use the convention that $D_0 = D$.
For $\widetilde{D}_t = D_t + t^{n+1} D_{n+1}$ to be a deformation, one more condition need to satisfy, namely,
\begin{center}
$D \circ D_{n+1} + D_1 \circ D_{n} + \cdots + D_{n} \circ D_1 + D_{n+1} \circ D = 0.$
\end{center}
In other words, $D_{n+1}$ must satisfy
\begin{center}
$ \delta_\alpha (D_{n+1}) = - \sum_{i=1}^n D_i \circ D_{n+1-i}.$
\end{center}
The right hand side of the above equation is called the obstruction to extend the deformation $D_t$ to a deformation of order $n+1$.

\begin{prop}
The obstruction is a Hochschild $3$-cocycle, that is,
\begin{center}
$ \delta_\alpha \big( - \sum_{i=1}^n D_i \circ D_{n+1-i} \big) = 0.$
\end{center}
\end{prop}
\begin{proof}
Note that $\{D, D_1, \ldots, D_n \}$ satisfy a system of relations (\ref{deform-rel}). Therefore, it follows that (by pre-compose and post-compose the $i$-th relation by $D_{n+1 -i}$)
\begin{align*}
A_i := D \circ D_i \circ D_{n+1-i} + D_1 \circ D_{i-1} \circ D_{n+1-i} + \cdots + D_{i-1} \circ D_1 \circ D_{n+1-i} + D_i \circ D \circ D_{n+1-i} = 0,\\
B_i := D_{n+1-i} \circ D \circ D_i +  D_{n+1-i} \circ D_1 \circ D_{i-1} + \cdots +  D_{n+1-i} \circ D_{i-1} \circ D_1 +  D_{n+1-i} \circ D_i \circ D = 0,
\end{align*}
for $i=1, 2, \ldots, n.$ Hence, we have
\begin{align}\label{deform-eqn-2}
\sum_{i=1}^n A_i - \sum_{i=1}^n B_i = 0.
\end{align}
By an easy observation (after cancelling terms) shows that
\begin{align*}
\sum_{i=1}^n A_i  -  \sum_{i=1}^n B_i =~& \sum_{i=1}^n D \circ D_i \circ D_{n+1-i}  -  \sum_{i=1}^n D_{n+1-i} \circ D_i \circ D \\
=~& [D, \sum_{i=1}^n D_i \circ D_{n+1-i} ] = \delta_\alpha \big( \sum_{i=1}^n D_i \circ D_{n+1-i} \big).
\end{align*}
Hence, the result follows from Equation (\ref{deform-eqn-2}).
\end{proof}

It follows from the above proposition that the obstruction defines a cohomology class in $H^3_\alpha (A, A)$. If this cohomology class is zero, then the obstruction is given by a coboundary (say $\delta_\alpha (D_{n+1})$). In other words, $\widetilde{D}_t = D_t + t^{n+1} D_{n+1}$ defines a deformation of order $n+1$.

In view of this, we get the following.
\begin{thm}\label{3-zero-extension}
If $H^3_\alpha (A, A) = 0$, every deformation of order $n$ can be extended to deformation of order $n+1$.
\end{thm}

\section{Strongly homotopy hom-Lie algebras}\label{sec7}
In this section, we first recall the definition of $HL_\infty$-algebra introduced in \cite{sheng-chen}. We give some other equivalent descriptions of $HL_\infty$-algebras. An appropriate skew-symmetrization of $HA_\infty$-algebras give rise to $HL_\infty$-algebras. Finally, we define module over $HL_\infty$-algebras and Chevalley-Eilenberg cohomology.

\subsection{$HL_\infty$-algebras}
\begin{defn}\label{hl}
	An $HL_\infty$-algebra is a graded vector space $L = \oplus L_i$ together with
	\begin{itemize}
		\item[(i)] a collection $\{ l_k | ~ 1 \leq k < \infty \}$ of skew-symmetric multilinear maps
		$l_k : L^{\otimes k} \rightarrow L$ with deg~$(l_k) = k-2$,
		\item[(ii)] a linear map $\alpha : L \rightarrow L$ of degree $0$ with
\begin{center}	
$\alpha \big(	l_k (a_1, \ldots, a_k)  \big) = l_k \big( \alpha (a_1), \ldots, \alpha (a_k) \big)$
\end{center}
	\end{itemize}
	such that for all $n \geq 1$,
	\begin{align}\label{hl1-iden}
	\sum_{i+j = n+1}^{} \sum_{\sigma}^{} \chi (\sigma) ~ (-1)^{i (j-1)} ~ l_j   \big(  l_i (a_{\sigma (1)}, \ldots, a_{\sigma (i)}), \alpha^{i-1} a_{\sigma (i+1)}, \ldots, \alpha^{i-1} a_{\sigma (n)} \big) = 0,
	\end{align}
	for all $a_1, \ldots, a_n \in L$, and $\sigma$ runs over all $(i, n-i)$ unshuffles with $i \geq 1$.
\end{defn}

An $HL_\infty$-algebra as above is denoted by $(L, l_k, \alpha)$. When $\alpha =$ identity, one gets the definition of an $L_\infty$-algebra. When $L$ is a vector space considered as a graded vector space concentrated in degree $0$, we get hom-Lie algebras \cite{hls}.

The above definition of an $HL_\infty$-algebra has the following consequences. For $n=1$, we have
$l_1^2 = 0,$
which means that the degree $-1$ map $l_1 : V \rightarrow V$ is a differential. Therefore, $(L, l_1)$ is a chain complex. If we write $l_2 = [-,-]$, then this bracket is graded skew-symmetric and satisfies (for $n=2$)
\begin{center}
$ l_1 [ a , b ] = [ l_1 (a) ,b ] + (-1)^{|a|} [ a, l_1(b)], ~\text{ for } a, b \in L.$
\end{center}
It says that the differential $l_1$ is a graded derivation for the bracket $[-,-]$. For $n=3$, we have
\begin{align*}
&(-1)^{|a||c|}[[a,b], \alpha(c)] +  (-1)^{|b||a|}  [ [b,c], \alpha(a)] +
 (-1)^{|c||b|} [[c,a], \alpha(b)] \\
&= (-1)^{|a||c| + 1} \bigg\{   l_1 \big(l_3(a,b,c)\big) + l_3 \big( l_1 (a), b, c \big) + (-1)^{|a|} l_3 \big( a, l_1 (b), c \big) + (-1)^{|a|+ |b|} l_3 \big( a,b, l_1 (c) \big)   \bigg\}.
\end{align*}
Therefore, the graded skew-symmetric product $l_2$ does not satisfy (in general) the graded hom-Jacobi identity.
However, it does satisfy up to a term involving $l_3$. Similarly, for higher $n$, we get higher coherence laws that $l_k$'s must satisfy. It is now easy to see that a graded hom-Lie algebra is an $HL_\infty$-algebra with $l_k = 0$, for $k \neq 2$, and, 
a DG hom-Lie algebra is an $HL_\infty$-algebra with $l_k = 0$, for $k \geq 3$.

It also follows from the above observation that the homology $H_* (L) = H_* (L, l_1)$ of an $HL_\infty$-algebra $(L, l_k, \alpha)$ carries a graded hom-Lie algebra.

\begin{exam}
Let $(L, l_k)$ be an $L_\infty$-algebra and $\alpha : L \rightarrow L$ be a degree $0$ map satisfying $\alpha \circ l_k = l_k \circ \alpha^{\otimes k}$, for all $k \geq 1$. Such a map $\alpha$ is called a strict morphism. If $\alpha: L \rightarrow L$ is a strict morphism, then $(L, \alpha^{k-1} \circ l_k, \alpha)$ is an $HL_\infty$-algebra.
\end{exam}

\begin{remark}
In \cite{sheng-chen} Sheng and Chen studied $2$-term $HL_\infty$-algebras in details. Motivated from the paper of Baez and Crans \cite{baez-crans}, they introduced hom-Lie $2$-algebras and show that the category of $2$-term $HL_\infty$-algebras and the category of hom-Lie $2$-algebras are equivalent \cite[Theorem 3.9]{sheng-chen}.
\end{remark}

\begin{remark}
In this respect, we also mention the following. Like any Lie algebra gives rise to a solution of the Yang-Baxter equation, any Lie $2$-algebra gives a solution of the `Zamolodchikov tetrahedron equation' (a categorification of the Yang-Baxter equation) \cite{baez-crans}. In the context of hom-type algebras, D. Yau proved that hom-Lie algebras give rise to a solution of the hom-Yang-Baxter equation (an hom analogue of the Yang-Baxter equation) \cite{yau2}. It is interesting to verify whether hom-Lie $2$-algebras give rise to a solution of a suitable hom analogue of the `Zamolodchikov tetrahedron equation'.
\end{remark}

Next, we introduce an equivalent notion of $HL_\infty [1]$-algebra.
\begin{defn}\label{hl1}
	An $HL_\infty [1]$-algebra structure on a graded vector space $V$ consists of degree $- 1$ symmetric linear maps $\varrho_k : V^{\otimes k} \rightarrow V$ and a linear map $\alpha : V \rightarrow V$ of degree $0$ with
	\begin{center}
	$ \alpha \big( \varrho_k ( v_1, \ldots, v_k ) \big) = \varrho_k \big( \alpha(v_1), \ldots, \alpha(v_k) \big)$
	\end{center}
	and for all $n \geq 1$,
	\begin{align}\label{hl1-id}
	\sum_{i + j = n + 1}^{} \sum_{\sigma \in Sh (i, n-i)}^{} \epsilon (\sigma)~ \varrho_{j} \big(  \varrho_i (v_{\sigma(1)}, \ldots, v_{\sigma (i)}), \alpha^{i-1} v_{\sigma(i+1)}, \ldots, \alpha^{i-1} v_{\sigma(n)}   \big) = 0, \text{ for } a_1, \ldots, a_n \in V.
	\end{align}
\end{defn}

Like classical case, an $HL_\infty$-structure is related to an $HL_\infty[1]$-structure by a degree shift. 
Let $(L, l_k, \alpha)$ be a $HL_\infty$-algebra. Take $V = s L$ and define
$\varrho_k : V^{\otimes k} \rightarrow V$ by
\begin{align}\label{hl-hl1-map}
\varrho_k = (-1)^{\frac{k(k-1)}{2}}~ s \circ l_k \circ (s^{-1})^{\otimes k}.
\end{align}
On the other hand $l_k$ can be reconstructed from $\varrho_k$ as $l_k = s^{-1} \circ \varrho_k \circ s^{\otimes k}.$ See \cite{lada-markl} for more details. Then we have the following.
\begin{prop}\label{hl-hl1-prop}
	An $HL_\infty$-structure on a graded vector space $L$ is equivalent to an $HL_\infty [1]$-structure on the graded vector space $V = sL$. 
\end{prop}

\begin{remark}
	Note that Definition \ref{hl} is natural in the sense that the differential $l_1$ has degree $-1$, the product $l_2$ preserves the degree and $l_n$'s are higher homotopies, for $n \geq 3$. However, one can see that Definition \ref{hl1} is simpler in sign.
\end{remark}

\subsection{Coderivation interpretation}
Let $V$ be a graded vector space. We now consider the free graded commutative algebra $S V = \oplus_{k \geq 0} S^k (V)$ on $V$, where $S^k(V)$ is the quotient of $V^{\otimes k}$ by the subspace genetared by elements of the form $v_{\sigma (1)} \otimes \cdots \otimes v_{\sigma (k)} - \epsilon(\sigma) v_1 \otimes \cdots \otimes v_k$. We denote the induced product on $SV$ by $\wedge$. There is also a coproduct on $SV$ defined by
\begin{align*}
\triangle (v_1 \wedge \cdots \wedge v_n) &= (1) \otimes  (v_1 \wedge \cdots \wedge v_n) \\
& + \sum_{i=1}^{n-1} \sum_{\sigma \in Sh(i,n-i)}^{} \epsilon (\sigma)~ ( v_{\sigma(1)} \wedge \cdots \wedge  v_{\sigma (i)}) \otimes ( v_{\sigma (i+1)} \wedge \cdots \wedge v_{\sigma(n)} )\\
& + (v_1 \wedge \cdots \wedge v_n) \otimes (1) . 
\end{align*}

If $\alpha : V \rightarrow V$ is a linear map of degree $0$, it induces a coalgebra map (denoted by the same symbol) $\alpha : S V \rightarrow S V$ defined by
\begin{center}
$\alpha (v_1 \wedge \cdots \wedge v_n) = \alpha (v_1) \wedge \cdots \wedge \alpha (v_n)$.
\end{center}

\begin{lemma}\label{skew-coderivation-lemma}
 Let $V$ be a graded vector space and $\alpha : V \rightarrow V$ be a linear map of degree $0$. Let $\varrho : S^n(V) \rightarrow V$ be a linear map $($of degree $|\varrho|)$ satisfying $\alpha \circ \varrho = \varrho \circ \alpha^{\wedge n}$, that is,
\begin{center}	
	$ \alpha ( \varrho (v_1 \wedge \cdots \wedge v_n)) = \varrho (   \alpha (v_1) \wedge \cdots \wedge \alpha (v_n) ).$
	\end{center}
	View $\varrho$ as a linear map $\varrho : S V \rightarrow V$ by letting its only non-zero component being given by the original $\varrho$ on $S^n (V)$. Then $\varrho$ lifts uniquely to a map  $\widetilde{\varrho} : {S V} \rightarrow {S V}$ $($of degree $|\varrho|)$ satisfying $\alpha \circ \widetilde{\varrho} = \widetilde{\varrho} \circ \alpha$  and 
\begin{center}	
	$\triangle \circ \widetilde{\varrho} = (\widetilde{\varrho} \otimes \alpha^{n-1} + \alpha^{n-1} \otimes \widetilde{\varrho}) \circ \triangle .$
	\end{center}
	More precisely, $\widetilde{\varrho}$ is given by
$$\widetilde{\varrho} (v_1 \wedge \cdots \wedge v_k) := \begin{cases} 0,~~&\text{if }k < n,\\
\sum_{\sigma \in Sh(n, k-n)}~\epsilon(\sigma) ~\varrho(v_{\sigma (1)} \wedge \cdots \wedge v_{\sigma (n)}) \wedge \alpha^{n-1} v_{\sigma(n+1)} \wedge \cdots \wedge \alpha^{n-1} v_{\sigma(k)}, &\text{if }k \geq n. \end{cases} $$
\end{lemma}

Let
\begin{center}
$\text{Coder}^p_\alpha (S V) := \{ D = \sum_{n \geq 1} \widetilde{\varrho_n} |~ \varrho_n : S^n V \rightarrow V \text{is a map of degree }p \text{ satisfying } \alpha \circ \varrho_n = \varrho_n \circ \alpha^{\otimes n} \}.$
\end{center}
One may also define a graded Lie algebra structure on $\text{Coder}^\bullet_\alpha (S V)$.
For $D = \sum_{n \geq 1} \widetilde{\varrho_n} \in \text{Coder}^p_\alpha (S V)$ and  $D' = \sum_{m \geq 1} \widetilde{\sigma_m} \in \text{Coder}^q_\alpha (S V)$, we define
$$D \circ D' = \sum_{n,m \geq 1} \widetilde{\varrho_n} \circ \widetilde{\sigma_m} \in \text{Coder}^{p+q}_\alpha (S V).$$
The graded Lie bracket on $\text{Coder}^\bullet_\alpha (S V)$ is then given by $[D,D'] = D \circ D' - (-1)^{pq} D' \circ D.$

\begin{prop}\label{coder-coder-zero-equiv-lie}
	Let $V$ be a graded vector space and $\alpha : V \rightarrow V$ be a linear map of degree $0$. Let $D = \sum_{n \geq 1} \widetilde{\varrho_n} \in Coder^{-1}_\alpha (S V)$. Then the condition $D \circ D = 0$
	is equivalent to the following equations:
	\begin{align*}
	\sum_{i + j = n + 1}^{} \sum_{\sigma \in Sh (i, n-i)}^{} \epsilon (\sigma)~ \varrho_{j} \big(  \varrho_i (v_{\sigma(1)}, \ldots, v_{\sigma (i)}), \alpha^{i-1} v_{\sigma(i+1)}, \ldots, \alpha^{i-1} v_{\sigma(n)}   \big) = 0, ~ \text{for all }n \geq 1.
	\end{align*}
\end{prop}

\medskip

In view of Proposition \ref{hl-hl1-prop} and Proposition \ref{coder-coder-zero-equiv-lie}, we get the following result.
\begin{thm}\label{hl-coder-thm}
	An $HL_\infty$-algebra structure on a graded vector space $L$ with respect to a linear map $\alpha : L \rightarrow L$ of degree $0$ is equivalent to an element $D \in Coder^{-1}_\alpha (S V)$ with $D \circ D = 0$, where $ V = sL$.
\end{thm}

It is known that the graded commutator of a differential graded hom-associative algebra gives rise to a differential graded hom-Lie algebra. More precisely, let $(A, \mu, \alpha, d)$ be a differential graded hom-associative algebra. Then $(A, [-,-], \alpha, d)$ is a differential graded hom-Lie algebra, where
\begin{center}
$[a, b] = \mu( a \otimes b) - (-1)^{|a||b|} \mu(b \otimes a)$
\end{center}
is the usual graded commutator. A more general theorem for strongly homotopy algebras is given by the following.

\begin{thm}\label{ha-comm-hl}
	Let $(A, \mu_k, \alpha)$ be an $HA_\infty$-algebra. Then $(A, l_k, \alpha)$ forms an $HL_\infty$-algebra, where the maps $l_k : \otimes^k A \rightarrow A$ are given by
\begin{center}	
	$ l_k (a_1, \ldots, a_k) = \sum_{\sigma \in S_k} \chi (\sigma) ~ \mu_k (a_{\sigma (1)}, \ldots, a_{\sigma (k)}),$
	\end{center}
	for $a_1, \ldots, a_k \in A$ and $1 \leq k < \infty.$ Here, $l_k$'s are appropriate skew-symmetrization of $\mu_k$'s.
\end{thm}

\begin{proof}
Our proof is similar to the proof of Lada and Markl \cite{lada-markl}.
Let the $HA_\infty$-algebra $(A, \mu_k, \alpha)$ is given by the coderivation $D \in \text{Coder}^{-1}_\alpha (TV)$, where $V = sA$. Then we have $D \circ D = 0$.

Consider the injective coalgebra map $Q : S V \rightarrow T V$ given by
\begin{center}
$ Q (v_1 \wedge \cdots \wedge v_n) = \sum_{\sigma \in S_n}^{} \epsilon (\sigma) (v_{\sigma (1)} \otimes \cdots \otimes v_{\sigma(n)}).$
\end{center}
Next, consider the linear map $\pi \circ D \circ Q : {S V} \rightarrow V$, where $\pi : TV \rightarrow V$ is the projection. Then by Lemma \ref{skew-coderivation-lemma}, we can extend it to a unique map ${\bf D} = \widetilde{\pi \circ D \circ Q} : SV \rightarrow SV$ such that ${\bf D} \in \text{Coder}^{-1}_\alpha (SV)$ and $\pi' \circ {\bf D} = \pi \circ D \circ Q$, where $\pi' : SV \rightarrow V$ is the projection. It is straightforward to verify that ${\bf D} = \sum_{k \geq 1} \widetilde{\varrho_k}$, where $\varrho_k$'s are defined from $l_k$'s as in (\ref{hl-hl1-map}). Thus, it remains to show that
${\bf D} \circ {\bf D} = 0$.

First we claim that $Q \circ {\bf D} = D \circ Q$. Since $Q$ is a coalgebra map, both $Q \circ {\bf D}$  and $D \circ Q$ are in $\text{Coder}^{-1}_\alpha ( S V, T V )$. Therefore, we only need to examine $ \pi \circ Q \circ {\bf D} = \pi \circ D \circ Q$.
Observe that $\pi \circ Q = \pi'$, hence, $\pi \circ Q \circ {\bf D} = \pi' \circ {\bf D}$. On the other hand, $\pi \circ D \circ Q = \pi' \circ {\bf D}$, by the definition of ${\bf D}$. Hence, we have $Q \circ {\bf D} = D Q$.

Thus, $\pi' \circ {\bf D} \circ {\bf D} = (\pi' \circ {\bf D}) \circ {\bf D} = (\pi \circ D  \circ Q) \circ {\bf D} = \pi \circ D \circ D \circ Q = 0$. Since, ${\bf D} \circ {\bf D} : {S V} \rightarrow {S V}$ is in $\text{Coder}^{-2}_\alpha (SV)$ and $\pi' \circ {\bf D} \circ {\bf D} = 0$, it follows that ${\bf D} \circ {\bf D} = 0$. Hence, the proof.
\end{proof}

\subsection{Cohomology and deformations}
\begin{defn}
	Let $(L, l_k, \alpha)$ be a $HL_\infty$-algebra. An $HL_\infty$-module is a graded vector space $M$ together with
	\begin{itemize}
		\item a collection $\{ \eta_k | ~ 1 \leq k < \infty \}$ of linear maps $\eta_k : \otimes^{k-1} L \otimes M \rightarrow M$ with deg $(\eta_k) = k-2,$
		\item a linear map $ \beta : M \rightarrow M$ of degree $0$ with
\begin{center}		
$ \beta \big( \eta_k (a_1, \ldots, a_{k-1}, m) \big) = \eta_k \big(  \alpha(a_1), \ldots, \alpha(a_{k-1}), \beta (m) \big)$
\end{center}
	\end{itemize}
such that for all $n \geq 1$,
	\begin{align}
		  \sum_{i+j= n+1}^{} \sum_{\sigma \in Sh (i, n-i)}^{} \chi (\sigma) (-1)^{i (j-1)}~ \eta_j \big(   \eta_i (a_{\sigma (1)}, \ldots, a_{\sigma (i)}) , \gamma^{i-1} a_{\sigma (i+1)}, \ldots, \gamma^{i-1} a_{\sigma (n)} \big) = 0 ,
	\end{align}
for $a_1, \ldots, a_{n-1} \in L$ and $a_n \in M$.
\end{defn}

Note that in the above definition we use the following conventions. For any $\sigma \in Sh (i, n-i),$ we have either $a_{\sigma (i)} = a_n$ or $a_{\sigma (n)} = a_n$. In the first case, we define
\begin{align*}
\eta_j \big(  \eta_i &(a_{\sigma (1)}, \ldots , a_{\sigma (i)}) , \gamma^{i-1} a_{\sigma (i+1)}, \ldots, \gamma^{i-1} a_{\sigma (n)} \big) \\
=~& (-1)^{j-1 + (i + \sum_{k=1}^{i} |a_{\sigma(k)}|)(\sum_{k=i+1}^{n} |a_{\sigma(k)}|)}~
\eta_j \big(  \alpha^{i-1} a_{\sigma (i+1)}, \ldots, \alpha^{i-1} a_{\sigma (n)} , \eta_i  (a_{\sigma (1)}, \ldots , a_{\sigma (i)}) \big).
\end{align*}
In the second case, we define
\begin{align*}
\eta_j \big(   \eta_i (a_{\sigma (1)}, \ldots, a_{\sigma (i)}) , \gamma^{i-1} a_{\sigma (i+1)}, \ldots, \gamma^{i-1} a_{\sigma (n)} \big) = \eta_j \big(   l_i (a_{\sigma (1)}, \ldots, a_{\sigma (i)}) , \alpha^{i-1} a_{\sigma (i+1)}, \ldots, \beta^{i-1} a_{\sigma (n)} \big).
\end{align*}

It is easy to observe that $M = L$ is an $HL_\infty$-module over itself with $\eta_k = l_k$ and $\beta = \alpha$. We now define the Chevalley-Eilenberg cohomology of an $HL_\infty$-algebra with coefficients in itself.

Let $(L, l_k, \alpha)$ be an $HL_\infty$-algebra with the corresponding coderivation is given by $D \in \text{Coder}^{-1}_\alpha{(S V)}$. The Chevalley-Eilenberg cochain complex of $L$ with coefficients in itself is given by $\big(  C^\bullet_\alpha (L,L), \delta_\alpha \big),$ where  $C^\bullet_\alpha (L,L) = \text{Coder}^{-(\bullet -1)}_\alpha (S V)$ and
\begin{center}
$ \delta_\alpha (-) := [D, -].$
\end{center}
Then $D \circ D = 0$ implies that $\delta_\alpha^2 = 0$. The cohomology groups are denoted by $H^\bullet_\alpha (L, L)$.

\medskip

Like $HA_\infty$-algebra case, one may also study the deformation of $HL_\infty$-algebras. Let $(L, l_k, \alpha)$ be an $HL_\infty$-algebra with the corresponding coderivation is given by $D \in \text{Coder}^{-1}_\alpha{(S V)}$.
A deformation of $L$ is given by a formal sum
\begin{center}
$D_t = D + tD_1 + t^2 D_2 + \cdots \in \text{Coder}^{-1}_\alpha (S V)[[t]]$
\end{center}
such that $D_t \circ D_t = 0$.
Similar to the case of $HA_\infty$-algebra, one can show that $H^2_\alpha (L,L) = 0$ implies the rigidity of the $HL_\infty$-structure and $H^3_\alpha (L,L)$ obstruct the extension of a deformation of order $n$ to a deformation of order $n+1$.

\begin{remark}
In \cite{fial-pen}, the authors consider a more general deformation of homotopy algebras $(A_\infty, L_\infty,\ldots)$ whose deformation parameter lies in a local algebra. Inspired from the classical deformations (associative, Lie,$\ldots$), they also considered miniversal deformation of homotopy algebras. In our context, one may also study miniversal deformation of $HA_\infty$- and $HL_\infty$-algebras.
\end{remark}

\section{Appendix}\label{sec8}
In this appendix, we introduce some generalizations of $HA_\infty$-algebras and $HL_\infty$-algebras.

\subsection{$HA_\infty$-category}
An $A_\infty$-category is a kind of category (but not a true category) in which the composition of morphisms is associative up to homotopy (see \cite{kell} for instance).
$A_\infty$-categories are category theoretical model of $A_\infty$-algebras. In this appendix, we introduce $HA_\infty$-category as a categorical model of $HA_\infty$-algebras.

\begin{defn}
An $HA_\infty$-category $\mathcal{A}$ consists of

$\bullet$ a class of objects {\sf obj}($\mathcal{A}$),

$\bullet$ for all $A, B \in$ {\sf obj}($\mathcal{A}$), there is a graded vector space $\text{Hom}_{\mathcal{A}} (A, B)$ and a linear map $$\alpha_{A, B} : \text{Hom}_{\mathcal{A}} (A, B) \rightarrow \text{Hom}_{\mathcal{A}} (A, B)$$
(also denoted by $\alpha$ when there is no confusion) of degree $0$,

$\bullet$ for all $k \geq 1$ and $A_0, A_1, \ldots, A_k \in$ {\sf obj}($\mathcal{A}$), a graded composition map
\begin{center}
$\mu_k : \text{Hom}_{\mathcal{A}} (A_{k-1}, A_k) \otimes \text{Hom}_{\mathcal{A}} (A_{k-2}, A_{k-1}) \otimes \cdots \otimes \text{Hom}_{\mathcal{A}} (A_0, A_1) \rightarrow \text{Hom}_{\mathcal{A}} (A_0, A_k)$
\end{center}
of degree $k-2$ satisfying 
\begin{center}
 $\alpha \circ \mu_k = \mu_k \circ \alpha^{\otimes k}$, for all $k \geq 1$,
\end{center}
and for all $n \geq 1$ and objects $A_0, A_1, \ldots, A_n \in$ {\sf obj}($\mathcal{A}$),
\begin{align*}
\sum_{i+j = n+1}^{} \sum_{\lambda =1}^{j} (-1)^{\lambda (i+1) + i (|a_1| + \cdots + |a_{\lambda -1 }|)} ~ \mu_{j} \big(  \alpha^{i-1}a_1, \ldots, \alpha^{i-1} a_{\lambda -1}, \mu_i ( a_{\lambda}, \ldots, a_{\lambda + i-1}), \ldots, \alpha^{i-1} a_n   \big) = 0,
\end{align*}
for $a_i \in \text{Hom}_{\mathcal{A}} (A_{n-i}, A_{n-i+1}),~ i=1, \ldots, n.$
\end{defn}

It is easy to see that if $\mathcal{A}$ is an $HA_\infty$-category with one object $\{ \star \}$, then the graded vector space $\text{Hom}_{\mathcal{A}} (\star, \star)$ inherits an $HA_\infty$-algebra structure. Therefore, an $HA_\infty$-category is a categorical model of an $HA_\infty$-algebra. Another class of examples are given by dg hom-categories in which $\mu_k = 0$, for all $k \geq 3$.

When $\alpha =$identity, one gets the definition of an $A_\infty$-category. Let $(\mathcal{A}, \mu_k)$ be an $A_\infty$-category and for all $A, B \in ${\sf obj}($\mathcal{A}$), there is a map
\begin{center}
$	\alpha : \text{Hom}_{\mathcal{A}} (A, B) \rightarrow \text{Hom}_{\mathcal{A}} (A, B)$
\end{center}
satisfying $ \alpha \circ \mu_k = \mu_k \circ \alpha^{\otimes k}$, for all $k \geq 1$. Such an $\alpha$ is called a strict functor. If $\alpha$ is a strict functor, then $(A, \alpha^{k-1} \circ \mu_k, \alpha)$ is an $HA_\infty$-category.

The categorical model of an $HA_\infty$-algebra morphism is given by $HA_\infty$-functor which can be define similarly.

\begin{remark}
Since $HA_\infty$-category is a category theoretical model of an $HA_\infty$-algebra, it would be interesting to extend some properties (like homotopy transfer theorem) of $HA_\infty$-algebra to $HA_\infty$-category. We also hope that there might exists an appropriate version of the Hochschild cohomology theory for some `good' $HA_\infty$-categories.
\end{remark}

\subsection{$HA_\infty$-coalgebras}
Like coalgebras are dual to algebras, hom-coalgebras are dual to hom-algebras \cite{makh-sil3}.
A hom-coalgebra is a triple $(A, \triangle, \alpha)$ consists of a vector space $A$ together with a linear map $\triangle : A \rightarrow A \otimes A$ and a linear map $\alpha : A \rightarrow A$ satisfying $\triangle \circ \alpha = \alpha^{\otimes 2} \circ \triangle$ and the following hom-coassociative identity
\begin{center}
$(\triangle \otimes \alpha) \circ \triangle = (\alpha \otimes \triangle) \circ \triangle.$
\end{center}
Given a coalgebra $(A, \triangle)$ and a coalgebra map $\alpha : A \rightarrow A$, one can associate a hom-coalgebra $(A, \triangle \circ \alpha , \alpha)$. The notions of graded hom-coalgebra and DG hom-coalgebra can be define similarly. Now, we introduce $HA_\infty$-coalgebra as a more general object where the hom-coassociative identity holds up to certain homotopy.

\begin{defn}
An $HA_\infty$-coalgebra is a graded vector space $A = \oplus A_i$ together with 

$\bullet$ a collection of maps $\triangle_k : A \rightarrow A^{\otimes k}$ of degree $k-2$, for $k \geq 1$,

$\bullet$ a linear map $\alpha : V \rightarrow V$ of degree $0$ satisfying 
\begin{center}
 $\alpha^{\otimes k} \circ \triangle_k =\triangle_k \circ \alpha,~~~$ for all $k \geq 1$,
 \end{center}
and for all $n \geq 1$,
$$\sum_{r + s+ t = n, r, t \geq 0, s \geq 1} (-1)^{rs + t}~ \big( (\alpha^{s-1})^{\otimes r} \otimes \triangle_s \otimes (\alpha^{s-1})^{\otimes t}   \big) \circ \triangle_{r+1+t} = 0.$$
\end{defn}

For $n=1$, it follows that $\triangle_1$ is a differential. For $n = 2$, we get that the differential $\triangle_1$ satisfies the co-Leibniz rule
$$\triangle_2 \circ \triangle_1 = (\triangle_1 \otimes \text{id}) \circ \triangle_2 + (\text{id} \otimes \triangle_1) \circ \triangle_2$$
with respect to the coproduct $\triangle_2$. For $n=3$, it follows that the coproduct $\triangle_2$ is hom-coassociative up to certain homotopy. Similarly, for higher $n$, we get higher coherence laws that $\triangle_k$'s must satisfy. It follows that an $HA_\infty$-coalgebra $(V, \triangle_k, \alpha)$ with $\triangle_k = 0$ for $k \geq 2$, is a DG hom-coalgebra.

The definition of morphism between $HA_\infty$-coalgebras can be define similarly. Finally, it would be interesting to prove homotopy transfer theorems for $HA_\infty$-coalgebras. One may also study $2$-term $HA_\infty$-coalgebras and prove results analogous to Subsections \ref{subsec-skeletal} and \ref{subsec-strict}. We hope that a categorical version of hom-coalgebras might be related to $2$-term $HA_\infty$-coalgebras.

\subsection{$HP_\infty$-algebras}
The notion of strongly homotopy Poisson algebras ($P_\infty$-algebras in short) was introduced by Cattaneo and Felder to study deformation of coisotropic submanifolds of a Poisson manifold \cite{cat-fel}. We now introduce strongly homotopy hom-Poisson algebras (or $HP_\infty$-algebras in short).

\begin{defn}
An $HP_\infty$-algebra is an $HL_\infty$-algebra $(L, l_k, \alpha)$ equipped with a graded commutative, hom-associative multiplication
\begin{center}
$\mu : L \otimes L \rightarrow L, ~~(a,b) \mapsto a \cdot b$
\end{center}
of degree $0$ such that each $l_k : L^{\otimes k} \rightarrow L$, $k \geq 1$, satisfies the following graded hom-Leibniz rule
\begin{align*}
l_k \big( \alpha (a_1), \ldots, \alpha (a_{k-1}), b \cdot c \big) =~& l_k (a_1, \ldots, a_{k-1}, b) \cdot \alpha(c) \\
&+ (-1)^{(k-2 + |a_1|+ \cdots + |a_{k-1}|)|b|} \alpha (b) \cdot l_k (a_1, \ldots, a_{k-1}, c),
\end{align*}
for all $a_1, \ldots, a_{k-1}, b, c \in L$.
\end{defn}
 An $HP_\infty$-algebra as above is denoted by $(L, l_k, \mu, \alpha).$ When $\alpha =$~identity, we recover $P_\infty$-algebras introduced in \cite{cat-fel}. When $L$ is a vector space (considered as a graded vector space concentrated in degree $0$), we get hom-Poisson algebras \cite{makh-sil}.

Given a $P_\infty$-algebra $(L, l_k, \mu)$ and a strict $P_\infty$-algebra morphism $\alpha$, one can define an $HP_\infty$-algebra $(L, \alpha^{k-1} \circ l_k, \alpha \circ \mu, \alpha)$.

\begin{remark}
Note that, in the above definition of an $HP_\infty$-algebra, we have assumed that the multiplication $\mu$ is hom-associative. In a more general notion of $HP_\infty$-algebra, one expects that the multiplication is hom-associative up to homotopy, in other words, the underlying graded vector space $L$ inherits an $HA_\infty$-algebra structure $(L, \mu_k, \alpha)$ as well. The graded hom-Leibniz rule can easily extend (for higher $\mu_k$'s) by using the Koszul sign convention.
\end{remark}
\noindent {\bf Acknowledgement.} The research was supported by the Institute post-doctoral fellowship of Indian Statistical Institute Kolkata. The author would like to thank the Institute for their support.

\end{document}